\documentclass[12pt] {article}
\usepackage{float}
\usepackage{delarray}
\usepackage{amsmath, latexsym, amsfonts, amssymb, amsthm, amscd,enumerate,hyperref}
\usepackage{graphicx}

\renewcommand{\leq}{\leqslant}
\renewcommand{\geq}{\geqslant}

\newcommand{\E}{\mathbb{E}}

\newcommand{\K}{\mathcal{K}}

\newcommand{\tr}{\operatorname{tr}}
\newcommand{\Tr}{\operatorname{Tr}}

\DeclareMathOperator{\id}{id}

\renewcommand{\phi}{\varphi}

\newtheorem{theorem}{Theorem}
\newtheorem{proposition}[theorem]{Proposition}
\newtheorem{remark}[theorem]{Remark}
\newtheorem{lemma}[theorem]{Lemma}

\newtheorem{corollary}[theorem]{Corollary}

\theoremstyle{definition}
\newtheorem{example}[theorem]{Example}
\newtheorem{definition}[theorem]{Definition}

\usepackage{delarray}
\usepackage{amsmath, latexsym, amsfonts, amssymb, amsthm, amscd}
\usepackage{delarray}
\usepackage{amsmath, mathdots, latexsym, amsfonts, amssymb, amsthm, amscd}
\usepackage{graphicx}
\usepackage{graphicx}
\usepackage{graphicx,,psfrag}
\usepackage{amssymb,amsmath,amsthm,enumerate}
\usepackage{color}
\newcommand{\C}{\mathbb{C}}

\newcommand{\1}{1\!\!{\sf I}}

\begin{document}
\title{Outlier eigenvalues for non-Hermitian  polynomials in independent i.i.d. matrices  and deterministic matrices}
\author{Serban Belinschi\thanks{Institut de Math\'ematiques de Toulouse; UMR5219; Universit\'e de Toulouse; CNRS;
 UPS, F-31062 Toulouse, FRANCE. 
E-mail: serban.belinschi@math.univ-toulouse.fr }, Charles Bordenave\thanks{Institut de Math\'ematiques de Marseille; CNRS; Aix-Marseille Universit\'e, Marseille, France. E-mail: charles.bordenave@univ-amu.fr},\\ Mireille Capitaine\thanks{Institut de Math\'ematiques de Toulouse; UMR5219; Universit\'e de Toulouse; CNRS;
 UPS, F-31062 Toulouse, FRANCE. 
E-mail: mireille.capitaine@math.univ-toulouse.fr }, Guillaume C\'ebron\thanks{Institut de Math\'ematiques de Toulouse; UMR5219; Universit\'e de Toulouse; CNRS;
 UPS, F-31062 Toulouse, FRANCE. 
E-mail: guillaume.cebron@math.univ-toulouse.fr }}
\maketitle
\begin{abstract}
We consider a square random matrix of size $N$ of the form $P(Y,A)$ where $P$ is a noncommutative polynomial, $A$ is a tuple of deterministic matrices converging in $\ast$-distribution, when $N$ goes to infinity, towards a tuple $a$ in some $\mathcal{C}^*$-probability space and   $Y$  is a tuple of independent matrices with i.i.d. centered entries
with variance $1/N$. We investigate the eigenvalues  of $P(Y,A)$ outside the spectrum of $P(c,a)$ where $c$ is a circular system which is free from $a$. We provide a sufficient condition to guarantee that these eigenvalues  coincide asymptotically with those of  $P(0,A)$.
\end{abstract}
\section{Introduction}
\subsection{Previous results}
Ginibre (1965) introduced the basic non-Hermitian ensemble of random matrix theory. A so-called Ginibre matrix is a  $N\times N$ matrix comprised of independent complex Gaussian entries.  More generally, an {\it i.i.d. random matrix} is a $N\times N$ random matrix $X_N=(X_{ij})_{1\leq i,j\leq N}$ whose entries are independent identically distributed complex entries with mean 0 and variance 1.\\
For any  $N\times N$ matrix $B$, denote by 
$\lambda_1(B), \ldots, \lambda_N(B)$
\noindent the  eigenvalues of $B$ and by $\mu_{B}$  the empirical spectral measure of $B$: $$\mu_{B} := \frac{1}{N} \sum_{i=1}^N \delta_{\lambda _{i}(B)}.$$  
 The following theorem is the culmination of the work of many authors \cite{Bai, BS10, G, GT, Me, PZ, TV0, TV2}.
\begin{theorem}
Let $X_N$ be an i.i.d. random matrix. Then the empirical spectral measure of $\frac{X_N}{\sqrt{N}}$ converges almost surely to the circular measure $\mu_c$ where $d\mu_c=\frac{1}{\pi} \1_{\vert z \vert \leq 1} dz.$

\end{theorem}

\noindent One can prove that when the fourth moment is finite, there are no significant outliers to the circular law.
\begin{theorem} $($see Theorem 1.4 in $\cite{Tout})$
Let $X_N$ be an i.i.d. random matrix whose entries have finite fourth moment: $\E(\vert X_{11}\vert^4)<+\infty.$ Then the spectral radius $\rho(\frac{X_N}{\sqrt{N}})=\sup_{1\leq j\leq N}\left| \lambda_j\left(\frac{X_N}{\sqrt{N}}\right)\right|$ converges to 1 almost surely as $N$ goes to infinity.
\end{theorem}

An addititive  low rank perturbation $A_N$ can create outliers outside the unit disk. Actually, when $A_N$ 
has bounded rank and bounded operator norm and the entries of the i.i.d. matrix have finite fourth 
moment, Tao proved that outliers outside the unit disk are stable in the sense that outliers of $M_N$ and $A_N$ coincide 
asymptotically.

\begin{theorem} (\cite{Tout})
Let $X_N$ be an i.i.d. random matrix whose entries have finite fourth moment.
 Let  $A_N$ be a deterministic matrix { with
rank $O(1)$} and operator norm $O(1)$. Let $\epsilon > 0$, and suppose that for all sufficiently large $N$, there are :
\begin{itemize}
\item { no
eigenvalues of $A_N$ in  $\{z \in \mathbb{C} : 1+\epsilon  < \vert z \vert  < 1+3 \epsilon \}$},\item {$ j =O(1)$ eigenvalues $\lambda_1(A_N), \ldots , \lambda_j(A_N)$
 in $\{z \in \mathbb{C} : \vert z \vert  \geq 1+3 \epsilon \}$. }
\end{itemize} Then, a.s , for sufficiently large $N$, there are
{ precisely j eigenvalues of $\frac{X_N}{\sqrt{N}} +A_N$
 in   $\{z \in \mathbb{C} : \vert z \vert  \geq 1+2 \epsilon \}$} 
and after labeling these eigenvalues properly, as $N$ goes to infinity,  for each $1 \leq i \leq j$, {$$\lambda_i( \frac{X_N}{\sqrt{N}} +A_N)=
\lambda_i(A_N)+o(1). $$} 
\end{theorem}

Two different ways of generalization of this result were subsequently considered.


Firstly, \cite{BC} investigated the same problem but dealing with full rank additive  perturbations. Main terminology related to free probability theory which is  used in the following is defined in Section \ref{freeproba} below.
Consider the deformed model:
\begin{equation}\label{modeleadditif} 
S_N=  A_N + \frac { X_N}{\sqrt N},
\end{equation}
where $A_N$ is a $N\times N$ deterministic matrix with operator norm $O(1)$ and such that $A_N\in
({\cal M}_N(\mathbb C),\tr_N)$ converges in $\ast$-moments to some noncommutative random variable $a$ 
in some ${\cal C}^*$-probability space $(\mathcal{A}, \phi)$. According to Dozier and Silverstein 
\cite{DozierSilver}, for any 
$z\in\mathbb C$, almost surely the empirical spectral measure  of 
${(S_N-zI_N){(S_N-zI_N)}^*}$ converges weakly towards a nonrandom distribution $\mu_z$ which is the 
distribution of $(c+a-z)(c+a-z)^*$ where $c$ is a circular operator which is free from $a$ in 
$(\mathcal{A},\phi)$.
\begin{remark}\label{supportloi} Note that for any operator $K $ in some ${\cal C}^*$-probability space $({\cal B}, \tau)$, $K$ is invertible if and only if $ KK^*$ and $K^*K$ are invertible. If $\tau$ is tracial,  the distribution $\mu_{KK^*} $ of $KK^*$ coincides with the distribution of $K^*K$.
Therefore, if  $\tau$ is faithful and tracial,   we have that  $0\notin \text{supp}(\mu_{KK^*}) $ if and only if $K$ is invertible.
\end{remark}
\noindent Therefore , since we can assume that $\phi$ is faithful and tracial, $\text{spect}
(c+a)=\{z\in\mathbb C\colon0\in\text{supp}(\mu_z)\},$ where  $\text{spect}$ denotes the spectrum. Actually, we will present some results of 
\cite{BC} only in terms of the spectrum of $c+a$ so that we do not need the assumption (A3) in \cite{BC} 
on the limiting empirical spectral measure of $S_N$. The authors in \cite{BC} gave a sufficient condition 
to guarantee that outliers of the deformed model \eqref{modeleadditif} outside the spectrum of $c+a$ 
are stable. For this purpose, they introduced the notion of well-conditioned matrix  which is related to 
the phenomenon of lack of outlier and of well-conditioned decomposition of $A_N$ which lead to the 
statement of a sufficient condition  for the stability of the outliers. We will denote by $s_1(B) \ge\cdots 
\geq s_N(B)$ the singular values of any $N \times N$ matrix $B$. For any set $K\subset \C$ and any $\epsilon>0$, 
$B(K,\epsilon)$ stands for the set $\{z \in\mathbb{C}\colon d(z,K)\leq \epsilon\}$.


\begin{definition}
Let $\Gamma \subset \mathbb{C} \setminus \text{spect}( c+a)$ be a compact set.  {$A_N$ is well-conditioned} in $\Gamma$ if  for any  $z \in \Gamma$,  there exists $\eta_z >0$ such that for all $N$ large enough, {$s_N(A_N - z I_N)> \eta_z$}. 
\end{definition}
\begin{theorem} (\cite{BC}) Assume that $A_N$ is well-conditioned in $\Gamma$,
 {Then,   a.s. for all $N$ large enough, $S_N$ has no eigenvalue in $\Gamma$.} 
\end{theorem}

\begin{corollary} (\cite{BC})
 {If  for any  $z \in \mathbb{C} \setminus \text{spect}( c+a)$},  there exists $\eta_z >0$ such that for all $N$ large enough, {$s_N(A_N - z I_N)> \eta_z$}, 
 {then,  for any $\varepsilon>0$,  a.s. for all $N$ large enough, all eigenvalues of $S_N$ are in $B(\text{spect}( c+a),\varepsilon)$}. 
\end{corollary}


 
Let us introduce now the notion of {\it well-conditioned decomposition} of $A_N$ which  allowed \cite{BC}  to exhibit  a sufficient condition for stability of outliers.


\begin{definition}
Let $\Gamma \subset \mathbb{C} \backslash \text{spect}( c+a)$ be a compact set.
 $A_N$ admits a 
{well-conditioned decomposition} if :
{$
A_N = A'_N + A''_N
$}
where
\begin{itemize}
\item There exists $M >0$ such that for all $N$, $\|A'_N\| + \|A''_N\| \leq M$.  
\item {For any $z\in\Gamma$, there exists $\eta_z>0$ such that for all $N$ large enough, $s_N(A'_N-zI_N)
>\eta_z$ (i.e $A'_N$ is well-conditioned in $\Gamma$) and $A''_N$ has  a fixed rank $r$.} 
\end{itemize}
\end{definition}


\begin{theorem}(\cite{BC}) \label{thmain}
Let $\Gamma \subset \mathbb{C} \backslash \text{spect}(c+a)$ be a compact set with continuous boundary.
Assume that $A_N$ admits a 
well-conditioned decomposition:
{$
A_N = A'_N + A''_N
$}.
  If for some $\varepsilon >0$ and all $N$ large enough,
\begin{equation}\label{eqratioAA'}
{\min_{z  \in \partial \Gamma} \left| \frac { \det ( A_N - z) }{\det ( A'_N -z ) } \right| \geq \varepsilon},
\end{equation}
then { a.s. for all $N$ large enough, the numbers of eigenvalues of $A_N$ and $S_N$ in $\Gamma$ are equal.} 
\end{theorem}

On the other hand, in \cite{CRW}, the authors investigate the outliers of several types of bounded rank perturbations of  the product of $m$ independent 
random matrices $X_{N,i}$, $i=1,\ldots,m$ with i.i.d entries. More precisely they study the eigenvalues outside the unit disk, of  the three following deformed models where $A_N$ and the $A_{N,j}$'s denote $N\times N$ deterministic matrices with rank $O(1)$ and norm $O(1)$:
\begin{enumerate}\item $\prod_{k=1}^m \left(\frac{X_{N,k}}{\sqrt{N}} +A_{N,k}\right)$;
\item the product, in some fixed order,
of the $m+s$ terms $\frac{X_{N,k}}{\sqrt{N}}, k=1,\ldots,m$, $(I_N+A_{N,j}), j=1,\ldots,s$;
\item  $\prod_{k=1}^m \frac{X_{N,k}}{\sqrt{N}} +A_N$. \end{enumerate}
Set ${\bf Y_N}=\left(\frac{X_{N,k}}{\sqrt{N}}, k=1,\ldots,m\right)$ and denote by 
${\bf A_N}$ the tuple of perturbations, that is ${\bf A_N}=\left(A_{N,k}, k=1,\ldots,m\right)$ in case 1., ${\bf A_N}=\left(A_{N,j}, j=1,\ldots,s\right)$ in case 2. and ${\bf A_N}=A_N$ in case 3..
In all cases 1.,2.,3., the model is some particular polynomial in ${\bf Y_N}$ and ${\bf A_N}$, let us say $P_{i}({\bf Y_N},{\bf A_N})$, $i=1,2,3$. It turns out that, according to  \cite{CRW}, for each $i=1,2,3$  the eigenvalues of  $P_{i}({\bf Y_N},{\bf A_N})$ and $P_{i}(0,{\bf A_N})$ outside the unit disk coincide asymptotically. Note that the unit disk is equal to the spectrum of each $P_i(c,0)$,  $i=1,2,3$ where $c$ is a free $m$-circular system.
\subsection{Assumptions and results}
In this paper we generalize the previous results from \cite{BC} to non-Hermitian  polynomials in several  
independent i.i.d. matrices  and deterministic matrices. Note that our results include in particular the
previous results from \cite{CRW}. Here are the matricial models we deal with. Let $t$ and $u$ be fixed 
nonzero integer numbers independent from $N$.

\begin{itemize}
\item[{\bf (A1)}]  $(A_N^{(1)},\ldots,A_N^{(t)})$ is a $t-$tuple of  $N\times N$   deterministic matrices   such that
\begin{enumerate}
\item
  for any $i=1,\ldots,t$, \begin{equation}\label{normedeA} \sup_N \Vert A_N^{(i)} \Vert< \infty,\end{equation} where $\Vert \cdot \Vert$ denotes the spectral norm,  
\item $ (A_N^{(1)}, \ldots, A_N^{(t)})$ converges in  $\ast$-distribution  towards a $t$-tuple $a=( a^{(1)},\dots, a^{(t)})$ in some ${\cal C}^*$-probability space $\left(  {\cal A},  \phi\right)$ where $\phi$ is faithful and tracial.
  \end{enumerate}

\item[{\bf (X1)}] We consider $u$ independent $N\times N$ random matrices $X_N^{(v)}=[X^{(v) }_{ij}]_{i,j=1}^N$, $v=1,\ldots,u$,  where, for each $v$,
$ [X^{(v)  }_{ij}]_{i\geq1,j\geq 1}$ is an infinite array of  random variables such that
$\{\sqrt{2}\Re(X^{(v)}_{ij})$,  $\sqrt{2} \Im(X^{(v)}_{ij}), i\geq 1,j\geq 1\}$ are independent identically distributed   centred random variables with variance 1 and finite fourth moment.
\end{itemize}

 Let $P$ be a  polynomial  in $t+u$ noncommutative indeterminates and define 
$$M_N=P\left(\frac{X_N^{(1)}}{\sqrt{N}}, \ldots,\frac{ X_N^{(u)}}{\sqrt{N}}, A_N^{(1)}, \ldots, A_N^{(t)}\right).$$
Note that we do not need any assumption on the convergence of the empirical spectral measure of $M_N$.
Let ${c}=(c^{(1)},\dots,c^{(u)})$ be a free noncommutative circular system in $\left({\cal A},\phi\right)$ 
which is free from ${a}=(a^{(1)},\dots,a^{(t)})$. According to the second assertion of Proposition 
\ref{pasde} below, for any $z\in\mathbb C$, almost surely, the empirical spectral measure of $(M_N-zI_N)
(M_N-zI_N)^*$ converges weakly to $\mu_z$ where $\mu_z$ is the distribution of $\left[P(c,a) -z1\right] 
\left[P( c,a) -z1\right]^*.$ Since we can assume that $\phi$ is faithful and tracial, we have by Remark \ref{supportloi} that
\begin{equation}\label{H2equivalent}
\text{spect}(P({c}, {a}))=\{z\in\mathbb C\colon0 \in \text{supp}(\mu_z)\}.
\end{equation}
Define  
$$
M_N^{(0)}=P(0_N, \ldots, 0_N, A_N^{(1)}, \ldots, A_N^{(t)}),
$$ where $0_N$ denotes the $N\times N$ null matrix.
Throughout the whole paper, we will call outlier any eigenvalue of $M_N$ or $M_N^{(0)}$ outside 
$\mathbb{C}\setminus\text{spect}(P({c}, {a}))$. We are now interested by describing the individual 
eigenvalues of $M_N$ outside $B(\text{spect}(P({c},{a})),\epsilon)$ for some $\epsilon>0$. To this end, we 
shall fix a set $\Gamma\subset\mathbb C$. In the lineage of \cite{BC}, our main result gives a sufficient 
condition to guarantee that outliers are stable in the sense that outliers of $M_N$ and $M_N^{(0)}$ 
coincide asymptotically.
\begin{theorem}\label{outlier} 
Assume that hypotheses ${\bf (A1), (X1)}$ hold. Let $\Gamma$ be  a compact subset of $\mathbb C
\setminus\mathrm{spect}( P(c,a))$. Assume moreover that \\

\noindent $ {\bf (A2)}\;$  for $k=1,\ldots,t$, $\;A_N^{(k)}=(A_N^{(k)})^{'}+(A_N^{(k)})^{''}$,\\
 
\noindent where $(A_N^{(k)})^{''}$ has a bounded rank  $r_k(N)=O(1)$ and $\left((A_N^{(1)})^{'},\dots,  (A_N^{(t)})^{'}
\right)$ satisfies\\
\begin{itemize}
\item $({\bf A_2'})$ for any $z$ in $\Gamma$, there exists $\eta_z>0$ such that for all $N$ large 
enough, there is no singular value of 
$$
P( 0_N, \ldots,0_N,  (A_N^{(1)})', \ldots, (A_N^{(t)})')-zI_N
$$ 
in $[0,\eta_z]$.
\item for any $k=1,\dots,t$, 
\begin{equation}\label{normedeAprime}
\sup_N \Vert(A_N^{(k)})^{'}\Vert <+\infty.
\end{equation}
\end{itemize}
If for some $\epsilon>0$, for all large $N$,
\begin{equation}\label{stabilite} 
\min_{z\in\partial\Gamma}\left|\frac{\det(zI_N-P(0_N,\dots,0_N,A_N^{(1)},\dots,A_N^{(t)})}{\det(zI_N-P(0_N, 
\dots,0_N,(A_N^{(1)})^{'},\dots,(A_N^{(t)})^{'})}\right|\geq\epsilon
\end{equation}
then almost surely for all large $N$, the numbers of eigenvalues of $M_N^{(0)}=P(0_N,\ldots,0_N,A_N^{(1)}, 
\ldots,A_N^{(t)})$ and $M_N=P\left(\frac{X_N^{(1)}}{\sqrt{N}},\dots,\frac{X_N^{(u)}}{\sqrt{N}},A_N^{(1)},\dots, 
A_N^{(t)}\right)$ in $\Gamma$ are equal.
\end{theorem}
The next statement is  an easy consequence of Theorem \ref{outlier}.
\begin{corollary}\label{cor:main} Assume that ${\bf (X1)}$ holds and that,  for $k=1,\ldots,t$, $A_N^{(k)}$ are deterministic matrices with
rank $O(1)$ and operator norm $O(1)$. Let $\epsilon >0$, and suppose that for all sufficiently large $N$, there are no
eigenvalues of
 $M_N^{(0)}=P(0,\ldots,0,A_N^{(1)}, 
\ldots,A_N^{(t)})$ in  $\{z \in \mathbb{C}, \epsilon  < d(z, \mathrm{spect}(P(c,0)))  < 3 \epsilon \}$, and there 
 are $j$ eigenvalues $ \lambda_1 (M_N^{(0)}), \ldots, \lambda_j (M_N^{(0)})$
 for some $j = O(1)$ in the region $\{z \in \mathbb{C}, d(z, \mathrm{spect}(P(c,0)))\geq 3 \epsilon \}$. 
 Then, a.s , for all large $N$, there are
{ precisely $j$ eigenvalues of $ M_N=P\left(\frac{X_N^{(1)}}{\sqrt{N}},\dots,\frac{X_N^{(u)}}{\sqrt{N}},A_N^{(1)},\dots, 
A_N^{(t)}\right)$
 in   $\{z \in \mathbb{C}, d(z, \mathrm{spect}(P(c,0))) \geq 2 \epsilon \}$} 
and after labeling these eigenvalues properly, {$$
\max_{j\in J}|\lambda_j(M_N)-\lambda_j (M_N^{(0)})|\to_{N\to+\infty} 0.
$$} 
\end{corollary}
We will first prove Theorem \ref{outlier} in the case $r =0$. 
\begin{theorem}\label{inclusion}
Suppose that assumptions of Theorem \ref{outlier} hold with, for any $k=1,\ldots,t,$ $(A_N^{(k)})'' = 0$, $A_N^{(k)} = (A^{(k)}_N)'$ and $\Gamma 
\subset\mathbb C\backslash\mathrm{spect} (P(c,a))$ a compact set. Then, a.s. for all $N$ large enough, 
$M_N$ has no eigenvalue in $\Gamma$. 

 In particular, if assumptions of Theorem \ref{outlier} hold with, for any $k=1,\ldots,t,$ $(A_N^{(k)})'' = 0$, $A_N^{(k)} = (A^{(k)}_N)'$   and $\Gamma =\mathbb C\backslash\mathrm{spect}(P(c,a))$ then for any $\varepsilon>0$, a.s. for all $N$ large enough, all eigenvalues of $M_N$ are in $B(\mathrm{spect} (P(c,a)) , \varepsilon)$. 
\end{theorem}

To prove Theorems \ref{inclusion} and  \ref{outlier}, we make use of a linearization procedure which brings the study of the polynomial back to that of the sum of matrices in a  higher dimensional space. Then, this allows us to follow the approach of \cite{BC}. But for this purpose, we need to establish substantial operator-valued free probability results. \\
In Section \ref{ex}, we present our theoretical results and corresponding simulations for four examples of random polynomial matrix models.
Section \ref{ovfp} provides required  definitions and preliminary results on  operator-valued free probability theory.
Section \ref{Sec:linny} describes the fundamental  linearization trick as introduced in 
\cite[Proposition 3]{A}. In Sections \ref{pp} and \ref{sp}, we establish Theorems  \ref{inclusion} and  \ref{outlier} respectively.

\section{Related results and examples}\label{ex}
Recall that we do not need any assumption on the convergence of the  empirical spectral measure of $M_N$. However, the convergence in $*$-distribution of $\left(\frac{X_N^{(1)}}{\sqrt{N}}, \ldots,\frac{ X_N^{(u)}}{\sqrt{N}}, A_N^{(1)}, \ldots, A_N^{(t)}\right)$ to $(c,a)=(c^{(1)}, \ldots, c^{(u)}, a^{(1)}, \ldots, a^{(t)})$ (see Proposition \ref{pasde}) implies the convergence in $*$-distribution of
$$M_N=P\left(\frac{X_N^{(1)}}{\sqrt{N}}, \ldots,\frac{ X_N^{(u)}}{\sqrt{N}}, A_N^{(1)}, \ldots, A_N^{(t)}\right)$$
to $P(c,a)$. In this situation, a good candidate to be the limit of the empirical spectral distribution of $M_N$ is the \emph{Brown measure} $\mu_{P(c,a)}$ of $P(c,a)$ (see \cite{Brown}). Unfortunately, the convergence of the empirical spectral distribution of $M_N$ to $\mu_{P(c,a)}$ is still an open problem for an arbitrary polynomial.

In the three following examples, we will consider the particular situation where we can decompose
$$
M_N=\alpha\frac{X_N^{(1)}}{\sqrt{N}}+Q\left(\frac{X_N^{(2)}}{\sqrt{N}}, \ldots,\frac{ X_N^{(u)}}{\sqrt{N}}, A_N^{(1)}, \ldots, A_N^{(t)}\right),
$$
with $\alpha>0$, $X_N^{(1)}$ a Ginibre matrix and $Q$ an arbitrary polynomial.
Indeed, in this case,
a beautiful result of \'{S}niady \cite{Snia} ensures that the empirical spectral distribution of $M_N$ converges to $\mu_{P(c,a)}$. Thus, the description of the limiting spectrum of $M_N$ inside $\text{supp}(\mu_{P(c,a)})$ is a question of computing explicitely $\mu_{P(c,a)}$ (a quite hard problem, which can be handled numerically by~\cite{BSS}). On the other hand, Theorem~\ref{outlier} explains the behaviour of the spectrum of $M_N$ outside $\text{spect}(P(c,a))$. Thus, we have a complete description of the limiting spectrum of $M_N$, except potentially in the set $\text{spect}(P(c,a))\setminus \text{supp}(\mu_{P(c,a)})$ which is not necessarily empty (even if it is empty in the majority of the examples known, see \cite{BL}).

For an arbitrary polynomial, we only know that any limit point of the empirical spectral distribution of $M_N$ is a balay\'{e}e of the measure $\mu_{P(c,a)}$ (see \cite[Corollary 2.2]{BL}), which implies that the support of any such limit point is contained in $\text{supp}(\mu_{P(c,a)})$, and in particular is contained in $\text{spect}(P(c,a))$.

\subsection{Example 1}We consider the matrix
\begin{align*}M_N&=P_1\left(\frac{X^{(1)}_N}{\sqrt{N}},\frac{X^{(2)}_N}{\sqrt{N}},\frac{X^{(3)}_N}{\sqrt{N}},A_N\right)\\
&=\frac{3}{2}\frac{X^{(1)}_N}{\sqrt{N}}+\frac{1}{6}\!\left(\frac{X^{(2)}_N}{\sqrt{N}}\right)^2\!\!A_N\!+\frac{1}{6}\frac{X^{(2)}_N}{\sqrt{N}}\frac{X^{(3)}_N}{\sqrt{N}}A_N\frac{X^{(3)}_N}{\sqrt{N}}+A_N^2\frac{X^{(3)}_N}{\sqrt{N}}+A_N+\frac{1}{8}A_N^2,\end{align*}
where $X^{(1)}_N,X^{(2)}_N,X^{(3)}_N$ are i.i.d. Gaussian matrices and$$A_N=\left(\begin{array}{ccccc}
2 &  &  &  &  \\ 
 & 2i &  &  &  \\ 
 & & 0 &  &  \\ 
 &  &  & \ddots &  \\ 
 & &  &  & 0
\end{array} \right).$$\begin{figure}[!h] 
\begin{center}
\includegraphics[width=9cm, trim = 0mm 2cm 0mm 2.5cm, clip]{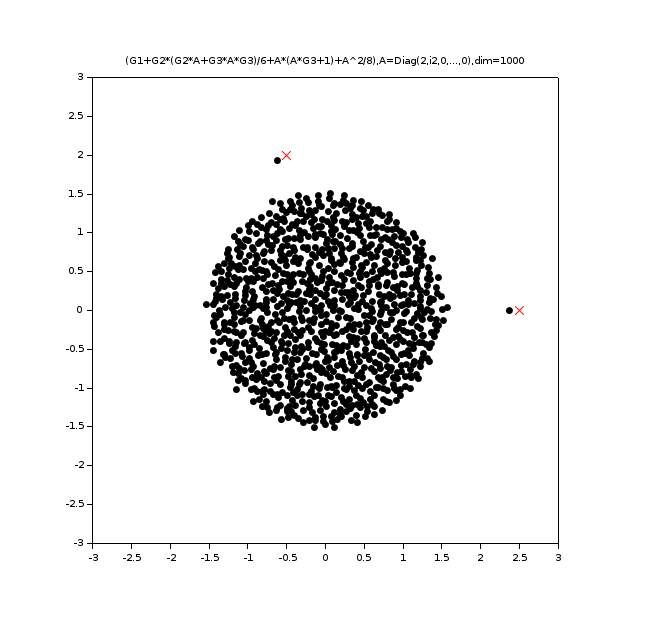}
\end{center}
\caption{In black, the eigenvalues of $P_1\left(\frac{X^{(1)}_N}{\sqrt{N}},\frac{X^{(2)}_N}{\sqrt{N}},\frac{X^{(3)}_N}{\sqrt{N}},A_N\right)$ for $N=1000$, and in red, the outliers $2.5$ and $2i-0.5
$ of
$P_1(0_N,0_N,0_N,A_N)$.}
\label{Example1} 
\end{figure}

The matrix $M_N$ converges in $*$-distribution to $\frac{3}{2}c$, where $c$ is a circular variable, and the empirical spectral measure  of $M_N$ converges to the Brown measure of $c$, which is the uniform law on the centered disk of radius $3/2$ by \cite{BL}. This disk is also  the spectrum of $\frac{3}{2}c$. Our theorem says that, outside this disk, the outliers of $M_N$ are closed to the eigenvalues $2.5$ and $2i-0.5
$ of
$P_1(0_N,0_N,0_N,A_N)=A_N+\frac{1}{8}A_N^2$ (see Figure~\ref{Example1}).

\subsection{Example 2}

We consider the matrix
\begin{align*}&M_N=P_2\left(\frac{X^{(1)}_N}{\sqrt{N}},\frac{X^{(2)}_N}{\sqrt{N}},\frac{X^{(3)}_N}{\sqrt{N}},A_N^{(1)},A_N^{(2)}\right)\\
&=\frac{1}{2}\frac{X^{(1)}_N}{\sqrt{N}}+\frac{1}{6}A_N^{(1)}\frac{X^{(2)}_N}{\sqrt{N}}\left(A_N^{(2)}+A_N^{(1)}+\frac{X^{(3)}_N}{\sqrt{N}}\right)\frac{X^{(2)}_N}{\sqrt{N}}+A_N^{(2)}\frac{X^{(3)}_N}{\sqrt{N}}A_N^{(1)}\\
&\quad+A_N^{(1)}+\frac{1}{2}A_N^{(2)},\end{align*}
where $X^{(1)}_N,X^{(2)}_N,X^{(3)}_N$ are i.i.d. Gaussian matrices,$$A_N^{(1)}=\left(\begin{array}{ccccc}
2 &  &  &  &  \\ 
 & -2.5 & &  &  \\ 
 &  & 0 &  &  \\ 
 &  &  & \ddots &  \\ 
 &  &  &  & 0
\end{array} \right)$$
and $A_N^{(2)}$ is a realization of a G.U.E. matrix.

\begin{figure}[!h] 
\begin{center}
\includegraphics[width=9cm, trim = 0mm 2cm 0mm 2.5cm,clip]{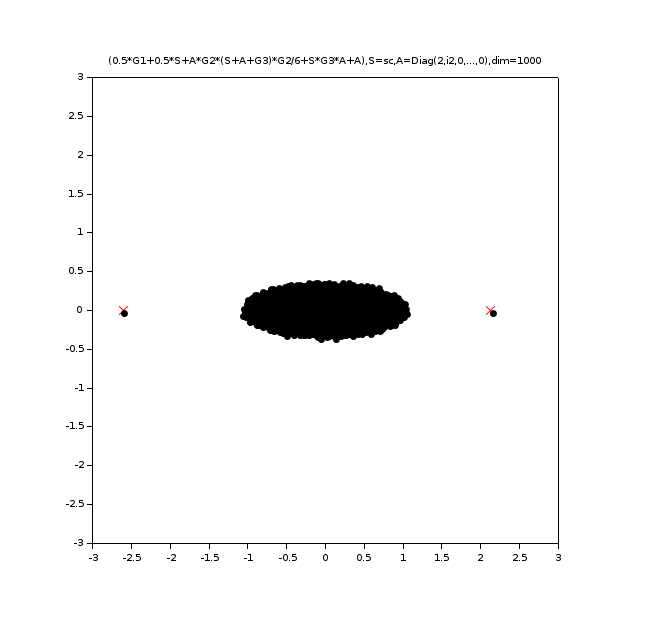}
\end{center}
\caption{In black, the eigenvalues of $P_2\left(\frac{X^{(1)}_N}{\sqrt{N}},\frac{X^{(2)}_N}{\sqrt{N}},\frac{X^{(3)}_N}{\sqrt{N}},A_N^{(1)},A_N^{(2)}\right)$ for $N=1000$, and in red, the limiting outliers $2.125$ and $-2.6$ of
$P_1(0,0,0,A_N^{(1)},A_N^{(2)})$.} 
\label{Example2}
\end{figure}

The matrix $M_N$ converges in $*$-distribution to the elliptic variable $\frac{1}{2}(c+s)$, where $c$ is a circular variable and $s$ a semicircular variable free from $c$. The empirical spectral measure  of $M_N$ converges to the Brown measure of $\frac{1}{2}(c+s)$, which is the uniform law on the interior of the ellipse $\{\frac{3}{2\sqrt{2}}\cos(\theta)+i\frac{1}{2\sqrt{2}}\sin(\theta):0\leq \theta < 2\pi\}$ by \cite{BL}. The interior of this ellipse is also  the spectrum of $\frac{1}{2}(c+s)$. Our theorem says that, outside this ellipse, the outliers of $M_N$ are closed to the outliers of
$P_2(0_N,0_N,0_N,A_N^{(1)},A_N^{(2)})=A_N^{(1)}+\frac{1}{2}A_N^{(2)}$ (see Figure~\ref{Example2}). Moreover, the outliers of $A_N^{(1)}+\frac{1}{2}A_N^{(2)}$ are those of an additive perturbation of a G.U.E. matrix, and converges to $2.125$ and $-2.6$ by \cite{Peche}.

\subsection{Example 3}
We consider the matrix
\begin{align*}M_N&=P_3\left(\frac{X^{(1)}_N}{\sqrt{N}},\frac{X^{(2)}_N}{\sqrt{N}},\frac{X^{(3)}_N}{\sqrt{N}},A_N^{(1)},A_N^{(2)}\right)\\
&=\frac{X^{(1)}_N}{\sqrt{N}}+A_N^{(1)}+A_N^{(2)}+A_N^{(1)}\frac{X^{(2)}_N}{\sqrt{N}}A_N^{(2)}\frac{X^{(2)}_N}{\sqrt{N}}+\frac{X^{(3)}_N}{\sqrt{N}}A_N^{(2)}\frac{X^{(2)}_N}{\sqrt{N}},\end{align*}
where $X^{(1)}_N,X^{(2)}_N,X^{(3)}_N$ are i.i.d. Gaussian matrices,$$A_N^{(1)}=\left(\begin{array}{cccccc}
1 &  &  &  & &  \\ 
 & \ddots &  &  & & \\ 
 &  & 1 &  & &  \\ 
 &  &  & -1&  &  \\ 
 &  &  & & \ddots &  \\ 
 &  &  & &  & -1
\end{array} \right)$$
is a matrix whose empirical spectral distribution converges to $\frac{1}{2}(\delta_1+\delta_{-1})$
and
$$A_N^{(2)}=\left(\begin{array}{ccccc}
1.5 &  &  &  &  \\ 
 & -2+2i &  &  &  \\ 
 &  & 0 &  &  \\ 
 &  &  & \ddots &  \\ 
 &  &  &  & 0
\end{array} \right)$$
\begin{figure}[!h]
\begin{center}
\includegraphics[width=9cm, trim = 0mm 2cm 0mm 2.5cm, clip]{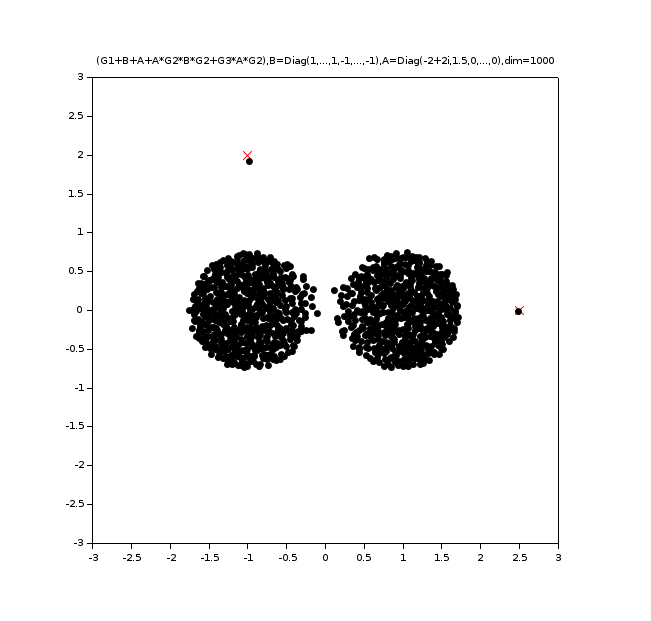}
\end{center}
\caption{In black, the eigenvalues of $P_3\left(\frac{X^{(1)}_N}{\sqrt{N}},\frac{X^{(2)}_N}{\sqrt{N}},\frac{X^{(3)}_N}{\sqrt{N}},A_N^{(1)},A_N^{(2)}\right)$ for $N=1000$, and in red, the outliers $2.5$ and $-1+2i$ of
$P_3(0_N,0_N,0_N,A_N^{(1)},A_N^{(2)})$.}
\label{Example3} 
\end{figure} 

The matrix $M_N$ converges in $*$-distribution to the random variable $c+a$, where $c$ is a circular variable and $a$ is a self-adjoint random variable, free from $c$, and whose distribution is $\frac{1}{2}(\delta_1+\delta_{-1})$. The empirical spectral measure  of $M_N$ converges to the Brown measure of $c+a$, which is absolutely continuous and whose support is the region inside the lemniscate-like curve in the complex plane with the equation $\{z\in \mathbb{C}:|z^2+1|^2=|z|^2+1\}$ by \cite{BL}. The interior of this ellipse is also  the spectrum of $c+a$. Our theorem says that, outside this ellipse, the outliers of $M_N$ are closed to the outliers $2.5$ and $-1+2i$ of
$P_3(0_N,0_N,0_N,A_N^{(1)},A_N^{(2)})=A_N^{(1)}+A_N{(2)}$ (see Figure~\ref{Example3}).

\subsection{Example 4}We consider the matrix
\begin{align*}M_N&=P_4\left(\frac{X^{(1)}_N}{\sqrt{N}},\frac{X^{(2)}_N}{\sqrt{N}},\frac{X^{(3)}_N}{\sqrt{N}},A_N\right)\\
&=\frac{1}{5}\left(\frac{X^{(1)}_N}{\sqrt{N}}+3I_N\right)\left(\frac{X^{(2)}_N}{\sqrt{N}}+A_N+2I_N\right)\left(\frac{X^{(3)}_N}{\sqrt{N}}+2I_N\right)-2I_N,\end{align*}
where $X^{(1)}_N,X^{(2)}_N,X^{(3)}_N$ are i.i.d. Gaussian matrices and$$A_N=\left(\begin{array}{ccccc}
2i &  &  &  &  \\ 
 & -2i &  &  &  \\ 
 &  & 0 &  &  \\ 
 &  &  & \ddots &  \\ 
 &  &  &  & 0
\end{array} \right).$$

\begin{figure}[!h] 
\begin{center}
\includegraphics[width=10cm, trim = 0mm 1cm 0mm 1.8cm, clip]{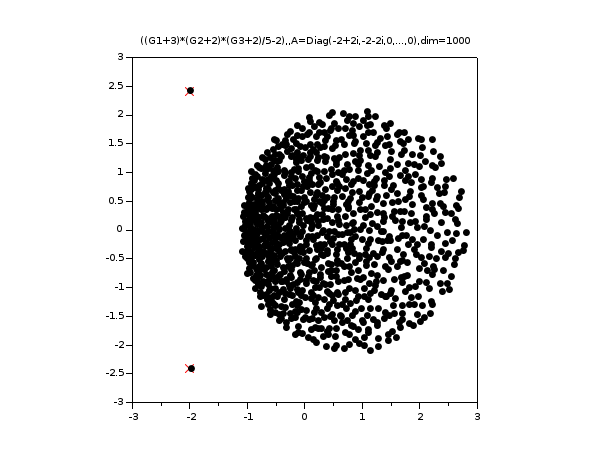}
\end{center}
\caption{In black, the eigenvalues of $P_4\left(\frac{X^{(1)}_N}{\sqrt{N}},\frac{X^{(2)}_N}{\sqrt{N}},\frac{X^{(3)}_N}{\sqrt{N}},A_N\right)$ for $N=1000$, and in red, the outliers $-2+2.4i$ and $-2-2.4i$ of
$P_4(0_N,0_N,0_N,A_N)$.} 
\label{Example4} 
\end{figure} 

The matrix $M_N$ converges in $*$-distribution to the random variable $(c_1+3)(c_2+2)(c_3+2)/5-2$, where $c_1,c_2,c_3$ are free circular variables. It is expected (but not proved) that the empirical spectral measure  of $M_N$ converges to the Brown measure of $\left(c_1+3\right)\left(c_2+2\right)\left(c_3+2\right)/5-2$. The spectrum of  $\left(c_1+3\right)\left(c_2+2\right)\left(c_3+2\right)/5-2$ is included in the set $\left(B(0,1)+3\right)(B(0,1)+2)\left(B(0,1)+2\right)/5-2$. Our theorem says that, outside this set, the outliers of $M_N$ are closed to the outliers $-2+2.4i$ and $-2-2.4i$ of
$P_4(0_N,0_N,0_N,A_N)=\frac{6}{5}A_N-2I_N$ (see Figure~\ref{Example4}).

\section{Free Probability Theory}\label{freeproba}
\subsection{{Scalar-valued free probability theory}}\label{sca}
For the reader's convenience, we recall the following basic definitions from free probability theory. For a thorough introduction to free probability theory, we refer to \cite{VDN}.
\begin{itemize}
\item A ${\cal C}^*$-probability space
is a pair $\left({\cal A}, \phi\right)$ consisting of a unital $ {\cal C}^*$-algebra ${\cal A}$
 and a state $\phi$ on ${\cal A}$ (i.e a linear map $\phi: {\cal A}\rightarrow \mathbb{C}$ such that $\phi(1_{\cal A})=1$ and $\phi(aa^*)\geq 0$ for all $a \in {\cal A}$)
 $\phi$ is a trace if it satisfies $\phi(ab)=\phi(ba)$ for every $(a,b)\in {\cal A}^2$. A trace is said to be faithful if $\phi(aa^*)>0$ whenever $a\neq 0$. 
An element of ${\cal A}$ is called a noncommutative random variable. 
\item The $\ast$-noncommutative distribution of a family $a=(a_1,\ldots,a_k)$ of noncommutative random variables in a ${\cal C}^*$-probability space $\left({\cal A}, \phi\right)$ is defined as the linear functional $\mu_a:P\mapsto \phi(P(a,a^*))$ defined on the set of polynomials in $2k$ noncommutative indeterminates, where $(a,a^*)$ denotes the $2k$-tuple $(a_1,\ldots,a_k,a_1^*,\ldots,a_k^*)$.
For any self-adjoint element $a_1$ in  ${\cal A}$,  there exists a probability  measure $\nu_{a_1}$ on $\mathbb{R}$ such that,   for every polynomial P, we have
$$\mu_{a_1}(P)=\int P(t) \mathrm{d}\nu_{a_1}(t).$$
Then,  we identify $\mu_{a_1}$ and $\nu_{a_1}$. If $\phi$ is faithful then the  support of $\nu_{a_1}$ is the spectrum of $a_1$  and thus  $\|a_1\| = \sup\{|z|, z\in \rm{support} (\nu_{a_1})\}$. 
\item A family of elements $(a_i)_{i\in I}$ in a ${\cal C}^*$-probability space  $\left({\cal A}, \phi\right)$ is free if for all $k\in \mathbb{N}$ and all polynomials $p_1,\ldots,p_k$ in two noncommutative indeterminates, one has 
\begin{equation}\label{freeness}
\phi(p_1(a_{i_1},a_{i_1}^*)\cdots p_k (a_{i_k},a_{i_k}^*))=0
\end{equation}
whenever $i_1\neq i_2, i_2\neq i_3, \ldots, i_{n-1}\neq i_k$ and $\phi(p_l(a_{i_l},a_{i_l}^*))=0$ for $l=1,\ldots,k$.
\item A   noncommutative random variable $x$  in a ${\cal C}^*$-probability space  $\left({\cal A}, \phi\right)$ is a standard  semicircular variable if
 $x=x^*$  and for any $k\in \mathbb{N}$, $$\phi(x^k)= \int t^k d\mu_{sc}(t)$$
where $d\mu_{sc}(t)=
\frac{1}{2\pi} \sqrt{4-t^2}{\1}_{[-2;2]}(t) dt$ is the semicircular standard distribution.
\item Let $k$ be a nonnull integer number. Denote by ${\cal P}$ the set of polynomials in $2k $ noncommutative indeterminates.
A sequence of families of variables $ (a_n)_{n\geq 1}  =
(a_1(n),\ldots, a_k(n))_{n\geq 1}$ in ${\cal C}^* $-probability spaces 
$\left({\cal A}_n, \phi_n\right)$ converges, when $n$ goes to infinity, respectively  in distribution if the map 
$P\in {\cal P} \mapsto
\phi_n(
P(a_n,a_n^*))$ converges pointwise.
\end{itemize}
\subsection{Operator-valued free probability theory}\label{ovfp}
\subsubsection{Basic definitions}
Operator-valued distributions and the operator-valued version of free probability were introduced by 
Voiculescu in \cite{ast} with the main purpose of studying freeness with amalgamation. Thus, an 
operator-valued noncommutative probability space is a triple $(M,E,B)$, where $M$ is a unital 
algebra over $\mathbb C$, $B\subseteq M$ is a unital subalgebra containing the unit of $M$, 
and $E\colon M\to B$ is a unit-preserving conditional expectation, that is, a linear $B$-bimodule 
map such that $E[1]=1$. We will only need the more restrictive context in which $M$ is a finite von 
Neumann algebra which is a factor, $B$ is a finite-dimensional von Neumann subalgebra of $M$ 
(and hence isomorphic to an algebra of matrices), and $E$ is the unique trace-preserving 
conditional expectation from $M$ to $B$. The $B$-valued distribution of an element $X\in M$ 
w.r.t. $E$ is defined to be the family of multilinear maps called the {\em moments} of $\mu_X$:
$$
\mu_X=\{B^{n-1}\ni(b_1,b_2,\dots,b_{n-1})\mapsto E[Xb_1Xb_2\cdots Xb_{n-1}X]\in B\colon n\ge0\},
$$
with the convention that the first moment (corresponding to $n=1$) is the element $E[X]\in B$, 
and the zeroth moment (corresponding to $n=0$) is the unit $1$ of $B$ (or $M$). The distribution
of $X$ is encoded conveniently by a noncommutative analytic transform defined for certain elements 
$b\in B$, which we agree to call the noncommutative Cauchy transform:
$$
G_X(b)=E\left[(X-b)^{-1}\right].
$$
(To be more precise, it is the noncommutative extension $G_{X\otimes1_n}(b)=
(E\otimes{\rm Id}_{{\cal M}_n(B)})\left[(X\otimes 1_n-b)^{-1}\right]$, for elements $B\in {\cal M}_n(B)$,
which completely encodes $\mu_X$ - see \cite{V}; since we do not need this extension, we
shall not discuss it any further, but refer the reader to \cite{V,Coalg,FAQ1,PV} for details.)
A natural domain for $G_X$ is the upper half-plane of $B$, $H^+(B)=\{b\in B\colon\Im b>0\}$.
It follows quite easily that $G_X(H^+(B))\subseteq H^+(B)$ - see \cite{FAQ1}. 

We warn here the reader that we have changed conventions in our paper compared to
\cite{Coalg,FAQ1,V}, namely we have chosen $G_X(b)=E\left[(X-b)^{-1}\right]$
instead of $E\left[(b-X)^{-1}\right]$, so that $G_X$ preserves $H^+(B)$.

Among many other results proved in \cite{ast}, one can find a central limit theorem for
random variables which are free with amalgamation. The central limit distribution is
called an {\em operator-valued semicircular}, by analogy with the free central limit
for the usual, scalar-valued random variables, which is Wigner's semicircular distribution.
It has been shown in \cite{ast} that an operator-valued semicircular distribution is 
entirely described by its operator-valued free cumulants: only the first and second
cumulants of an operator-valued semicircular distribution may be nonzero (see
also \cite{Mem,V}). For our purposes, we use the equivalent description of an 
operator-valued semicircular distribution via its noncommutative Cauchy
transform, as in \cite{HRS}: $S$ is a $B$-valued semicircular if and only if
$$
G_S(b)=\left(m_1-b-\eta(G_S(b))\right)^{-1},\quad b\in H^+(B),
$$
for some $m_1=m_1^*\in B$ and completely positive map $\eta\colon B\to B$. 
In that case, $m_1=E[S]$ and $\eta(b)=E[SbS]-E[S]bE[S]$. The above
equation is obviously a generalization of the quadratic equation determining 
Wigner's semicircular distribution: $\sigma^2G_S(z)^2+(z-m_1)G_S(z)+1=0$. Here
$m_1$ is the - classical - first moment of $S$, and $\sigma^2$ its classical variance,
which, as a linear completely positive map, is the multiplication with a positive constant. 
Unless otherwise specified, we shall from now on assume our semicirculars to be centered, i.e. $m_1=0$. 

\begin{example}\label{sem}
A rich source of examples of operator-valued semicirculars comes in the case of finite
dimensional $B$ from scalar-valued semicirculars: assume that $s_{i,j},1\le i\le j\le n$
are scalar-valued centered semicircular random variables of variance one. We do not 
assume them to be free. Then the matrix 
$$
\begin{bmatrix}
\alpha_1s_{1,1} & \gamma_{1,2}s_{1,2} & \gamma_{1,3}s_{1,3} &\cdots & \gamma_{1,n-1}s_{1,n-1} & \gamma_{1,n}s_{1,n} \\
\overline{\gamma_{1,2}}s_{1,2} & \alpha_2s_{2,2}& \gamma_{2,3}s_{2,3} &\cdots & \gamma_{2,n-1}s_{2,n-1} & \gamma_{2,n}s_{2,n} \\
\overline{\gamma_{1,3}}s_{1,3} & \overline{\gamma_{2,3}}s_{2,3} & \alpha_3s_{3,3} &\cdots & \gamma_{3,n-1}s_{3,n-1} & \gamma_{3,n}s_{3,n} \\
\vdots &\vdots &\vdots &\ddots & \vdots &\vdots \\
\overline{\gamma_{1,n-1}}s_{1,n-1} &\overline{\gamma_{2,n-1}}s_{2,n-1} & \overline{\gamma_{3,n-1}}s_{3,n-1} & \cdots & \alpha_{n-1}s_{n-1,n-1} &\gamma_{n-1,n}s_{n-1,n}\\
\overline{\gamma_{1,n}}s_{1,n} & \overline{\gamma_{2,n}}s_{2,n} & \overline{\gamma_{3,n}}s_{3,n} &
\cdots & \overline{\gamma_{n-1,n}}s_{n-1,n} & \alpha_ns_{n,n}
\end{bmatrix},
$$
where $\alpha_1,\dots,\alpha_n\in[0,+\infty)$ and $\gamma_{i,j}\in\mathbb C$, $1\le i<j\le n$, 
is an ${\cal M}_n(\mathbb C)$-valued semicircular. Note that we do allow our scalars to be zero. This is a 
particular case of a result from \cite{ieee}, and its proof can be found in great detail in \cite{MS}.
\end{example}

An important fact about semicircular elements, both scalar- and operator-valued, is that the sum of
two free semicircular elements is again a semicircular element (this follows from the fact that
a semicircular is defined by having all its cumulants beyond the first two equal to zero - see
\cite{ast}). In particular, if $\{s_{1,1}^{(1)},s_{1,2}^{(1)},s_{2,2}^{(1)},s_{1,1}^{(2)},
s_{1,2}^{(2)},s_{2,2}^{(2)}\}$ are centered all semicirculars of variance one, and in addition we 
assume them to be free from each other, then $\begin{bmatrix} s_{1,1}^{(1)} & s_{1,2}^{(1)}\\
s_{1,2}^{(1)} & s_{2,2}^{(1)}\end{bmatrix}$ and $\begin{bmatrix} s_{1,1}^{(2)} & is_{1,2}^{(2)}\\
-is_{1,2}^{(2)} & s_{2,2}^{(2)} \end{bmatrix}$ are ${\cal M}_2(\mathbb C)$-valued semicirculars which are
free over ${\cal M}_2(\mathbb C)$, so their sum $\begin{bmatrix} s_{1,1}^{(1)}+s_{1,1}^{(2)} & s_{1,2}^{(1)}+
 is_{1,2}^{(2)}\\s_{1,2}^{(1)}- is_{1,2}^{(2)} & s_{2,2}^{(1)}+s_{2,2}^{(2)}\end{bmatrix}$ is also an 
${\cal M}_2(\mathbb C)$-valued semicircular, despite its off-diagonal elements not being anymore 
distributed according to the Wigner semicircular law. This is hardly surprising: the two matrices
we have added are the limits of the real and imaginary parts of a G.U.E. (Gaussian Unitary Ensemble). 
The upper right corner of a G.U.E. is known to be a C.U.E. (Circular Unitary Ensemble), and its eigenvalues
converge to the uniform law on a disk. On the other hand, direct analytic computations show that
the sum $s_{1,2}^{(1)}\pm is_{1,2}^{(2)}$, with $s_{1,2}^{(1)}$ and $s_{1,2}^{(2)}$ free from each 
other, has precisely the same law. Thus, the following definition, due to Voiculescu, is natural.

\begin{definition}
An element $c$ in a ${}^*$-noncommutative probability space $(\mathcal A,\varphi)$ is called a {\em
circular random variable} if  $(c+c^*)/\sqrt{2}$ and $(c-c^*)/\sqrt{2}i$, respectively, 
are free from each other and identically distributed according to standard Wigner's semicircular law.
\end{definition}

\subsubsection{Preliminary results}\label{prelim}
We first establish preliminary results in free probability theory that we  will need in the following sections. 
\begin{lemma}\label{transfertliberte} Let $\{m_p^{(j)}, p=1,\ldots,4, j=1,\ldots,t\}$ be noncommutative random variables in some noncommutative probability space $(\mathcal{A}, \phi)$.
Let $s_i^{(1)}$,  $s_i^{(2)}$,$i=1, \ldots,u$  be   semicircular variables and $c_i$,$i=1, \ldots,u$ be    circular variables such that
$ s_1^{(1)}, \ldots, s_u^{(1)}, s_1^{(2)}, \ldots, s_u^{(2)}, c_1,\ldots, c_u, \{m_p^{(j)}, p=1,\ldots,4, j=1,\ldots,t\}$ are $\ast$-free in $(\mathcal{A}, \phi)$. Define for $i=1,\ldots,u$, $${\bf s_i}=\frac{1}{\sqrt{2}}\begin{pmatrix} s_i^{(1)} & c_i\\c_i^*  &  s_i^{(2)} \end{pmatrix}, $$ and  for $j=1,\ldots,t$, $${\bf m_j}=\begin{pmatrix} m_1^{(j)} & m_2^{(j)}\\m_3^{(j)}  &  m_4^{(j)}\end{pmatrix}.$$ Then, in the scalar-valued probability space $({\cal M}_2(\mathcal{A}), \tr_2\otimes \phi)$, ${\bf s_1}, \ldots, {\bf s_u},  \{{\bf m_j}, j=1,\ldots,t\}$ are free and 
for $i=1,\ldots,u$, each ${\bf s_i}$ is a semicircular variable. \end{lemma}
\begin{proof}Let us prove that ${\bf s_1}, \ldots, {\bf s_u}$ is free from $\mathcal M_2(\mathcal{B})$, where $\mathcal{B}$ is the $\ast$-algebra generated by $\{m_p^{(j)}, p=1,\ldots,4, j=1,\ldots,t\}$. We already now  (see \cite[Chapter 9]{MS}) that ${\bf s_1}, \ldots, {\bf s_u}$ are semicircular variables over $\mathcal M_2(\mathbb{C})$ which are free from $\mathcal M_2(\mathcal{B})$, with respect to $id_2\otimes \phi$. Moreover, the covariance mapping of ${\bf s_1}, \ldots, {\bf s_u}$ is the function $(\eta^{\mathcal M_2(\mathbb{C})}_{i,j}:{\cal M}_2(\mathbb{C})\to {\cal M}_2(\mathbb{C}))_{1\leq i,j \leq u}$, which can be computed as follows: for all $m=\left(\begin{array}{cc}
m_1 & m_2 \\ 
m_3 & m_4
\end{array} \right)\in {\cal M}_2(\mathbb{C}),$
we have
\begin{align*}\eta^{\mathcal M_2(\mathbb{C})}_{i,j}(m)&=(\id_2\otimes \phi)({\bf s}_{i}m{\bf s}_{j})\\
&=\frac{1}{2}\left(\begin{array}{c|c}
\phi(s_i^{(1)}m_1s_j^{(1)}) +\phi(s_i^{(1)} m_2 c_j^*)+\phi(c_i m_3s_j^{(1)})+ \phi(c_i m_4c_j^*)& \star \\ 
\hline  \star & \star \\ 
\end{array} \right)\\
&=\frac{\delta_{ij}}{2}\left(\begin{array}{c|c}
m_1 +m_4& 0\\ 
\hline 0&m_1 +m_4 \\ 
\end{array} \right)\\
&=\delta_{ij}\tr_2(m)I_2.
\end{align*}
Using \cite[Theorem 3.5]{NSS}, the freeness of ${\bf s_1}, \ldots, {\bf s_u}$ from ${\cal M}_2(\mathcal{B})$ over ${\cal M}_2(\mathbb{C})$ gives us the free cumulants of ${\bf s_1}, \ldots, {\bf s_u}$ over ${\cal M}_2(\mathcal{B})$. More concretely, we get that ${\bf s_1}, \ldots, {\bf s_u}$ are semicircular variables over ${\cal M}_2(\mathcal{B})$, with a covariance mapping $(\eta^{{\cal M}_2(\mathcal{B})}_{i,j}:{\cal M}_2(\mathcal{B})\to {\cal M}_2(\mathcal{B}))_{1\leq i,j \leq u}$ given by $\eta^{{\cal M}_2(\mathcal{B})}_{i,j}=\eta^{{\cal M}_2(\mathbb{C})}_{i,j}\circ (id_2\otimes \phi) $.

Because of the previous computation, we know that $\eta^{{\cal M}_2(\mathbb{C})}_{i,j}=\tr_2\circ\eta^{{\cal M}_2(\mathbb{C})}_{i,j}\circ \tr_2$, which means that $\eta^{{\cal M}_2(\mathcal{B})}_{i,j}=(\tr_2\otimes \phi)\circ \eta^{{\cal M}_2(\mathcal{B})}_{i,j}\circ (\tr_2\otimes \phi)$. As a consequence, using again \cite[Theorem 3.5]{NSS}, ${\bf s_1}, \ldots, {\bf s_u}$ are semicircular variables over $\mathbb{C}$ free from $M_2(\mathcal{B})$ with respect to $(\tr_2\otimes \phi)$, and the covariance mapping $\eta^{\mathbb{C}}_{i,j}$ is given by the restriction of the covariance mapping $\eta^{{\cal M}_2(\mathbb{C})}$ to $\mathbb{C}$: for all $m\in \mathbb{C}$
$$\eta^{\mathbb{C}}_{i,j}(m)=\delta_{ij}m,$$
which means that ${\bf s_1}, \ldots, {\bf s_u}$ are free standard semicircular variables.
\end{proof}

\begin{lemma}\label{symm}
Let $y$ be a noncommutative random variable in ${\cal M}_m({\cal A})$ and $c^{(1)},\dots,c^{(u)}$ be free circular 
variables in ${\cal A}$ free from the entries of $y$. Then, in the operator-valued probability space $({\cal M}_m({\cal A}), id_m\otimes \phi)$, 
$|\sum_{j=1}^u\zeta_j\otimes c^{(j)}+y|^2$ has the same distribution as $|\sum_{j=1}^u\zeta_j\otimes s_j+(I_m\otimes \epsilon) \cdot y|^2$ where $\epsilon$ is a selfadjoint $\{-1,+1\}$-Bernoulli variable in ${\cal A}$, independent from the entries of $y$, and $s_1,\ldots,s_u$ are free semicircular variables in ${\cal A}$, free from $\epsilon$ and the entries of $y$.
\end{lemma}

In the lemma above, we consider the symmetric version $\epsilon y$ of $y$, thanks to a noncommutative random variable $\epsilon$ which is \emph{tensor}-independent from the entries of $y$, in the sense that $\epsilon$ commutes with the entries of $y$ and $\phi(p_1(\epsilon)p_2(y_{i,j},y^*_{i,j}:i,j))=\phi(p_1(\epsilon))\phi(p_2(y_{i,j},y^*_{i,j}:i,j))$ for all polynomials $p_1,p_2$.

\begin{proof}Let $n\geq 0$. We compute the $n$-th moment of $| \sum_{j=1}^u \zeta_j \otimes c^{(j)} + y|^2 $ with respect to $id_m\otimes \phi$, and compare it to the $n$-th moment of $| \sum_{j=1}^u \zeta_j \otimes s_j + \epsilon y|^2 $ with respect to $id_m\otimes \phi$.

\noindent Let us set $a_0=y$ and $a_j= \zeta_j \otimes c^{(j)}$. We compute
\begin{align*}&id_m\otimes \phi(| \sum_{j=1}^u \zeta_j \otimes c^{(j)} + y|^{2n})\\
&=\sum_{0\leq i_1,\ldots,i_{2n}\leq u} id_m\otimes \phi(a_{i_1}a_{i_2}^* a_{i_3} a_{i_4}^* \ldots a_{i_{2n-1}}a_{i_{2n}}^*).
\end{align*}
Similarly,
\begin{align*}&id_m\otimes \phi(| \sum_{j=1}^u \zeta_j \otimes s_j + (I_m\otimes \epsilon) \cdot y|^{2n})\\
&=\sum_{0\leq i_1,\ldots,i_{2n}\leq u} id_m\otimes \phi(b_{i_1}b_{i_2}^* b_{i_3} b_{i_4}^* \ldots b_{i_{2n-1}}b_{i_{2n}}^*).
\end{align*}
where $b_0=(I_m\otimes \epsilon) \cdot y$ and $b_j=\zeta_j \otimes s_j$. In order to conclude, it suffices to prove that, for all $0\leq i_1,\ldots,i_{2n}\leq u$,
$$ id_m\otimes \phi(a_{i_1}a_{i_2}^* a_{i_3} a_{i_4}^* \ldots a_{i_{2n-1}}a_{i_{2n}}^*)=id_m\otimes \phi(b_{i_1}b_{i_2}^* b_{i_3} b_{i_4}^* \ldots b_{i_{2n-1}}b_{i_{2n}}^*).$$
Let us fix $0\leq i_1,\ldots,i_{2n}\leq u$.
Note that $a_0$ is free over ${\cal M}_m(\mathbb{C})$ from $a_j$ with
respect to $id_m\otimes \phi$ (see \cite[Chapter 9]{MS}). Let us fix $S=\{j:i_j\neq 0\}\subset \{1,\ldots,2n\}$ and use the moment cumulant formula (see \cite[page 36]{Mem}):
\begin{align*}&id_m\otimes \phi(a_{i_1}a_{i_2}^* a_{i_3} a_{i_4}^* \ldots a_{i_{2n-1}}a_{i_{2n}}^*)\\
&=\sum_{\pi\in NC(S)} (\hat{c}\cup \hat{\phi})(\pi \cup \pi^c)(a_{i_1}\otimes a_{i_2}^* \ldots a_{i_{2n-1}}\otimes a_{i_{2n}}^*)
\end{align*}
where $\pi^c$ is the largest partition of $S^c$ such that $\pi \cup \pi^c$ is noncrossing and $\hat{c}$ and $\hat{\phi}$ are the ${\cal M}_m(\mathbb{C})$-valued cumulant function and the ${\cal M}_m(\mathbb{C})$-valued moment function associated to the conditional expectation $id_m\otimes \phi$. We use here the notation of \cite[Notation 2.1.4]{Mem} which defines $(\hat{c}\cup \hat{\phi})(\pi \cup \pi^c)$ as some ${\cal M}_m(\mathbb{C})$-valued multiplicative function that acts on the blocks of $\pi$ like $\hat{c}$ and on the blocks of $\pi^c$ like $\hat{\phi}$.

Recall that the cumulants of $\zeta_j \otimes c^{(j)}$ are vanishing if $\pi$ is not a pairing and if $\pi$ is not alternating (which means that $\pi$ links two indices with the same parity). Now, let us remark that if $\pi$ is a pairing which is alternating, then $\pi^c$ is even (each blocs of $\pi^c$ is even). Thus,
\begin{align*}&id_m\otimes \phi(a_{i_1}a_{i_2}^* a_{i_3} a_{i_4}^* \ldots a_{i_{2n-1}}a_{i_{2n}}^*)\\
&=\sum_{\substack{\pi\in NC(S)\\\pi \text{ pairing and alternating}}} (\hat{c}\cup \hat{\phi})(\pi \cup \pi^c)(a_{i_1}\otimes a_{i_2}^* \ldots a_{i_{2n-1}}\otimes a_{i_{2n}}^*)\\
&=\sum_{\substack{\pi\in NC(S)\\\pi \text{ pairing and alternating}\\\pi^c \text{ even}}} (\hat{c}\cup \hat{\phi})(\pi \cup \pi^c)(a_{i_1}\otimes a_{i_2}^* \ldots a_{i_{2n-1}}\otimes a_{i_{2n}}^*).
\end{align*}

Similarly, the cumulants of $\zeta_j \otimes s^{(j)}$ are vanishing if $\pi$ is not a pairing and that the moment of $b_0$ is vanishing if $\pi^c$ is  odd. Moreover, if $\pi$ is a pairing and $\pi^c$ is even, then $\pi$ is alternating. As a consequence,
\begin{align*}&id_m\otimes \phi(b_{i_1}b_{i_2}^* b_{i_3} b_{i_4}^* \ldots b_{i_{2n-1}}b_{i_{2n}}^*)\\
&=\sum_{\pi\in NC(S)} (\hat{c}\cup \hat{\phi})(\pi \cup \pi^c)(b_{i_1}\otimes b_{i_2}^* \ldots b_{i_{2n-1}}\otimes b_{i_{2n}}^*)\\
&=\sum_{\substack{\pi\in NC(S)\\\pi \text{ pairing}\\\pi^c \text{ even}}} (\hat{c}\cup \hat{\phi})(\pi \cup \pi^c)(b_{i_1}\otimes b_{i_2}^* \ldots b_{i_{2n-1}}\otimes b_{i_{2n}}^*)\\
&=\sum_{\substack{\pi\in NC(S)\\\pi \text{ pairing and alternating}\\\pi^c \text{ even}}} (\hat{c}\cup \hat{\phi})(\pi \cup \pi^c)(b_{i_1}\otimes b_{i_2}^* \ldots b_{i_{2n-1}}\otimes b_{i_{2n}}^*).
\end{align*}
In order to conclude, it suffices to remark that $\epsilon y$ and $y$ has the same even ${\cal M}_m(\mathbb{C})$-valued moments and $\zeta_j \otimes c^{(j)}$ and $\zeta_j \otimes s^{(j)}$ has the same alternating ${\cal M}_m(\mathbb{C})$-valued cumulants.
\end{proof}

It follows from \cite{CAOT} that the support in ${\cal M}_m(\mathbb C)^{\rm sa}$ of the addition of   a semicircular $s$  of variance $\eta$
and a  selfadjoint noncommutative random variable $y\in ({\cal M}_m({\cal A}),id_m\otimes \phi)$ which is free with amalgamation over ${\cal M}_m(\C)$ with $s$,  is given via its complement in 
terms of $y$ and the functions 
\begin{equation}\label{funk}
H(w)=w-\eta(G_y(w)) \text{ and }\omega(b)=b+\eta(G_y(\omega(b)),
\end{equation}
where $G_x(b) =id_m\otimes \phi\left[(x-b)^{-1}\right]$.
\begin{proposition}\label{2}
If $w\in {\cal M}_m(\mathbb C)^{\rm sa}$ is such that $y-w$ is invertible and $\text{spect}(\eta\circ G_y'(w))
\subset\overline{\mathbb D}\setminus\{1\}$, then $s+y-H(w)$ is invertible. Conversely, if $b\in 
{\cal M}_m(\mathbb C)^{\rm sa}$ is such that $s+y-b$ is invertible, then $y-\omega(b)$ is invertible. 
\end{proposition}

It follows quite easily that $\text{spect}(\eta\circ G_y'(\omega(b)))\subset\overline{\mathbb D}$.
Generally, all conditions on the derivatives of $\omega$ and $H$ follow from the two functional
equations above.

\begin{proof} Assume that $y-w$ is invertible and $\text{spect}(\eta\circ G_y'(w))\subset\overline{\mathbb 
D}\setminus\{1\}$. Since $w=w^*$, the derivative $G'_y(w)$ is completely positive, so $\eta\circ G_y'(w)$
is completely positive. This means according to \cite[Theorem 2.5]{EH} that the spectral
radius $r$ of $\eta\circ G_y'(w)$ is reached at a positive element $\xi\in M_m(\mathbb C)$, so that
necessarily $r\ge0$. Since $1\not\in\sigma(\eta\circ G_y'(w))$ by hypothesis, it follows that $r<1$,
and thus
$$
\text{spect}(\eta\circ G_y'(w))\subseteq r\overline{\mathbb D}\subsetneq\mathbb D.
$$
This forces the derivative of $H(w)$, $H'(w)=\mathrm{Id}_{M_m(\mathbb C)}-\eta\circ G'_y(w))$,
to be invertible as a linear operator from ${\cal M}_m(\mathbb C)$ to itself. By the inverse function
theorem, $H$ has an analytic inverse on a small enough neighborhood of $H(w)$ onto a 
neighborhood of $w$. Since $H$ preserves the selfadjoints near $w$, so must the inverse.
On the other hand, the map $v\mapsto H(w)+\eta(G_y(v))$ sends the upper half-plane into itself
and has $w$ as a fixed point. Since its derivative has all its eigenvalues included {\em strictly}
in $\mathbb D$ (recall that the spectral radius $r<1$), it follows that $w$ is actually an attracting
fixed point for this map. Since for any $b$ in the upper half-plane, $\omega(b)$ is given as the
attracting fixed point of $v\mapsto b+\eta(G_y(v))$, it follows that $\omega$ coincides with the
local inverse of $H$ on the upper half-plane, so the local inverse of $H$ is the unique analytic 
continuation of $\omega$ to a neighborhood of $H(w)$.  This proves that $\omega$ extends
analytically to a neighborhood of $H(w)$ and the extension maps selfadjoints from this neighborhood
to ${\cal M}_m(\mathbb C)^{\rm sa}$. In particular, $\omega(H(v))=v$ and $G_{s+y}(H(v))=
G_y(\omega(H(v)))=G_y(v)$ are selfadjoint for all $v=v^*$ in a small enough neighborhood of $w$,
showing that $s+y-H(w)$ is invertible.

Conversely, say $b=b^*$ and $s+y-b$ is invertible. Then $G_{s+y}$ is analytic on a neighborhood
of $b$ and maps selfadjoints from this neighborhood into ${\cal M}_m(\mathbb C)^{\rm sa}$. Since
$\omega(b)=b+\eta(G_{s+y}(b))$, the same holds for $\omega$. Since, by \cite[Proposition 4.1]{CAOT},
$\text{spect}(\omega'(v))\subset\{\Re z>1/2\}$ for any $v$ in the upper half-plane, the analyticity of 
$\omega$ around $b=b^*$ implies $\text{spect}(\omega'(0))\subset\{\Re z\ge1/2\}$. Thus, $\omega$ is 
invertible wrt composition around zero by the inverse function theorem. As argued above, $H$ is its
inverse, and extends analytically to a small enough neighborhood of $\omega(b)$, with selfadjoint
values on the selfadjoints. Composing with $H$ to the left in Voiculescu's subordination relation
$G_{s+y}(v)=G_y(\omega(v))$ yields $G_{y+s}(H(w))=G_y(w)$, guaranteeing that $G_y$ is analytic
on a neighborhood of $\omega(b)$, with selfadjoint values on the selfadjoints, and so $y-\omega(b)$
must be invertible.
\end{proof}

\begin{remark}\label{meh}
The proof of the previous proposition, based on \cite[Theorem 2.5]{EH}, makes the condition
$\text{spect}(\eta\circ G'_y(0))\subseteq\overline{\mathbb D}\setminus\{1\}$ equivalent to the existence
of an $r\in[0,1)$ such that $\text{spect}(\eta\circ G'_y(0))\subseteq r\overline{\mathbb D}$.
\end{remark}

The following lemma is a particular case of the above proposition.

\begin{lemma}\label{Serban}  Consider the operator-valued $\cal C^*$-algebraic
noncommutative probability space $({\cal M}_m({\mathcal A}),id_m\otimes \phi, {\cal M}_m(\mathbb{C}))$ and a pair of selfadjoint random variables $s,y\in {\cal M}_m({\mathcal A})$ which are free over ${\cal M}_m(\mathbb{C})$ with
respect to $id_m\otimes \phi$. Assume that $s$ is a centered semicircular of variance $\eta\colon {\cal M}_m(\mathbb{C})\to {\cal M}_m(\mathbb{C})$ and that each entry of  $y \in {\cal M}_m({\mathcal A})$ is a noncommutative symmetric random variable in $({\mathcal A}, \phi)$.
We define $G_x(b)=id_m\otimes \phi\left[(x-b)^{-1}\right]$. Then  $s+y$ is invertible if and only if 
$0\not\in\text{spect}(y)$ and  $\text{spect}(\eta\circ G'_y(0))$ is included in $\overline{\mathbb D}
\setminus\{1\}$.
\end{lemma}

\begin{proof}
Note that our hypotheses that all entries of the selfadjoint $y$ are symmetric and that $s$ is centered 
imply automatically that $H(i{\cal M}_m(\mathbb C)^+)\subseteq i{\cal M}_m(\mathbb C)^{\rm sa}$ and 
$\omega(i{\cal M}_m(\mathbb C)^+)\cup G_y(i{\cal M}_m(\mathbb C)^+)\cup G_{y+s}(i{\cal M}_m(\mathbb C)^+)\subseteq i{\cal M}_m(\mathbb C)^+.$ 

Assume that $y$ is invertible and $\text{spect}(\eta\circ G_y'(0))\subseteq\overline{\mathbb D}
\setminus\{1\}$. In particular, $G_y$ is analytic on a neighborhood of zero in ${\cal M}_m(\mathbb C)$.
Proposition \ref{2} implies that $s+y-H(0)$ is invertible. Since $H(i{\cal M}_m(\mathbb C)^+)
\subseteq i{\cal M}_m(\mathbb C)^+$, it follows from the formula of $H$ that $H(0)=0$. Thus, $s+y$ is
invertible.

Conversely, assume that $s+y$ is invertible, so that $G_{s+y}$ extends analytically to a
small neighborhood of zero in such a way that it maps selfadjoints to selfadjoints.
Since $\omega(b)=b+\eta(G_{s+y}(b))$, it follows that $\omega$ does the same. 
According to Proposition \ref{2}, $y-\omega(0)$ is invertible.
Since $\omega(i{\cal M}_m(\mathbb C)^+)\subseteq i{\cal M}_m(\mathbb C)^+$, we again have that
$\omega(0)=0$, so that $y$ is invertible.
\end{proof}
\section{Linearization trick}\label{Sec:linny}
A powerful tool to deal with noncommutative polynomials in random matrices or in operators is 
the so-called ``linearization trick.'' Its origins can be found in the theory of automata and formal 
languages (see, for instance, \cite{Sch}), where it was used to conveniently represent certain categories
of formal power series. In the context of operator algebras and random matrices, this procedure goes 
back to Haagerup and Thorbj{\o}rnsen \cite{HT05,HT06}   (see \cite{MS}).  We use the version from
\cite[Proposition 3]{A}, which has several advantages for our purposes, to be described below.\\
~~

We denote by $\mathbb C\langle X_1,\dots,X_k\rangle$ the complex $\ast$-algebra of polynomials in
$k$ noncommuting indeterminates $X_1,\dots,X_k$. The adjoint operation is given by the anti-linear 
extension of $(X_{i_1}X_{i_2}\cdots X_{i_l})^* = X_{i_l}^*\cdots X_{i_2}^*X_{i_1}^*$, $(i_1, \ldots, i_l)\in
\{1,\ldots,k\}^l, l\in \mathbb{N}\setminus\{0\}$. We will sometimes assume that some, or all, 
of the indeterminates are selfadjoint, i.e. $X_j^*=X_j$. Unless we make this assumption explicitly, the 
adjoints $X_1^*,\dots,X^*_k$ are assumed to be algebraically free from each other and from $X_1,
\dots,X_k$. 

Given a polynomial $P\in\mathbb C\langle X_1,\dots,X_k\rangle$, we call {\it linearization} of $P$ any $L_P\in {\cal M}_m(\mathbb{C}) \otimes \mathbb{C} \langle X_1,\ldots, X_k \rangle$  such that 
 $$L_P := \begin{pmatrix} 0 & u^*\\v & Q \end{pmatrix} \in {\cal M}_m(\mathbb{C}) \otimes \mathbb{C} \langle X_1,\ldots, X_k \rangle$$
where
\begin{enumerate}
\item $ m \in \mathbb{N}$,
\item $ Q \in {\cal M}_{m-1}(\mathbb{C})\otimes \mathbb{C} \langle X_1,\ldots, X_k \rangle$ is invertible in
the complex algebra ${\cal M}_{m-1}(\mathbb{C})\otimes \mathbb{C} \langle X_1,\ldots, X_k \rangle$,
\item $u^*$ is a row vector and $v$ is a column vector, both of length $m-1$, with
entries in $\mathbb{C} \langle X_1,\ldots, X_k \rangle$,
\item  the polynomial entries in $Q, u$ and $v$ all have degree $\leq 1$,\\
\item \label{linny}
$\hspace{4cm} {P=-u^*Q^{-1}v}.$
\end{enumerate}

We refer to Anderson's paper \cite{A} for the - constructive - proof of the existence of a linearization 
$L_P$ as described above for any given polynomial $P\in\mathbb C\langle X_1,\dots,X_k\rangle$. It 
turns out that if $P$ is selfadjoint, then $L_P$ can be chosen to be self-adjoint.\\
The well-known result about Schur complements yields then the following invertibility equivalence.  

\begin{lemma}\label{inversible}\cite[Chapter 10, Corollary 3]{MS}
Let $P\in\mathbb{C}\langle X_1,\dots,X_k\rangle$ and let $L_P \in {\cal M}_m(\mathbb{C}\langle X_1,
\dots,X_k\rangle)$ be a linearization of $P$ with the properties outlined above. Let $e_{11}$ be the 
$m\times m$ matrix whose single nonzero entry equals one and occurs in the row 1 and column 1. Let 
$y = (y_1,\dots, y_k)$ be a $k$-tuple of  operators in a unital ${\cal C}^*$-algebra ${\cal A}$. 
Then, for any $z\in \mathbb{C}$, $ze_{11}\otimes 1_{\cal A}-L_P(y)$ is invertible if and only if $z 1_{\cal 
A}-P(y)$ is invertible and  we have 
\begin{equation}\label{coin}
\left(ze_{11}\otimes 1_{{\cal A}} - 
L_P(y)\right)^{-1}=\begin{pmatrix}\left(z1_{\cal A}-P(y)\right)^{-1} & \star\\\star & \star \end{pmatrix}.
\end{equation}
\end{lemma}
\begin{lemma}\label{resHari} Let $P\in\mathbb{C}\langle X_1,\dots,X_k\rangle$ and let $L_P \in {\cal M}_m(\mathbb{C}\langle X_1,
\dots,X_k\rangle)$ be a linearization of $P$ with the properties outlined above.
There exist two polynomials $T_1$ and $T_2$  in $k$ commutative indeterminates, with nonnegative coefficients, depending only on $L_P$, such that, for any  
$k$-tuple $y = (y_1,\dots, y_k)$  of  operators in a unital ${\cal C}^*$-algebra ${\cal A}$, for any $z\in \C$ such that $z1_{\cal A}-P(y)$ is invertible,
\begin{eqnarray}
\left\|(ze_{11}\otimes 1_{\cal A}\!-\!L_P(y))^{-1}\right\|&\leq &T_1\!\left(\!\|y_1\|,\dots,\|y_k\|\!\right)
\times\left\|(z1_{\cal A}-P(y))^{-1}\right\| \nonumber
\\&&
+ T_2\left( \|y_1\|, \ldots, \|y_k\|\right).\label{Hari}\end{eqnarray}
\end{lemma}
\begin{proof}
The linearization of P can be written as 
$$
L_P=\begin{bmatrix}
0 & u^*\\
v & Q
\end{bmatrix}\in {\cal M}_m(\mathbb C\langle X_1,\dots,X_{k}\rangle)
$$
Now, a matrix calculation
in which we suppress the  variable $y$ shows that
\\

$
(ze_{11}\otimes 1_{\cal A} -L_P)^{-1}$ $$=
\begin{bmatrix}
1_{\cal A}  & 0\\
-Q^{-1}v & I_{(m-1)}\otimes 1_{\cal A}
\end{bmatrix}
\begin{bmatrix}
(z-P)^{-1}  & 0\\
0&-Q^{-1}
\end{bmatrix}
\begin{bmatrix}
1_{\cal A}  & -u^*Q^{-1}\\
0 & I_{(m-1)}\otimes 1_{\cal A}
\end{bmatrix}.
$$
Since $v$, $u^*$, and $Q^{-1}$ 
are polynomials in $  y_1,\ldots,y_k$, the result readily follows. \end{proof}
In Section \ref{serbanlinearisation}, we will provide an explicit construction of a linearization that is best
adapted to our purposes. In this construction, it is clear that we can always find a linearization such that, for any $k$-tuple $y$ of matrices,
\begin{equation}
\label{detQ}\det Q(y)=\pm 1.
\end{equation}

\section{No outlier; proof of Theorem \ref{inclusion}}\label{pp} 
By Bai-Yin's theorem (see \cite[Theorem 5.8]{BS10}), there exists $C>0$ such that, almost surely for all 
large $N$, $\|M_N\|\leq C$, so that for  the first assertion of Theorem \ref{inclusion} readily yields the 
second one, by choosing 
$$
\Gamma=\{z\in\mathbb C,d(z,\text{spect}(P(c,a)))\ge\epsilon, |z|\le C\}.
$$
Remember that, by \eqref{H2equivalent},
$\text{spect}(P({c}, {a}))=\{z\in\mathbb C\colon0 \in \text{supp}(\mu_z)\},$ where $\mu_z$ is the distribution of $(P(c,a)-z)(P(c,a)-z)^*$.
The first assertion of Theorem \ref{inclusion} is equivalent to the following.
\begin{proposition}\label{nonunif} Let $\Gamma$ be a compact set of 
$\{z, 0 \notin \text{supp}(\mu_z)\}$; assume that   for any $z$ in $\Gamma$,   there exists $\eta_z>0$ such that   for all $N$ large enough, $$s_N\left(P( 0_N, \ldots,0_N,  (A_N^{(1)}), \ldots, (A_N^{(t)}))-zI_N\right)>\eta_z.$$ 
Then,  for any $z$ in $\Gamma$,  there exists $\gamma_z>0$, such that almost surely, for all large $N$,  $s_N(M_N-zI_N) \geq \gamma_z$. 
Consequently,  there
exists 
$\gamma_\Gamma > 0$ such that almost surely, for all large $N$,  $\inf_{z\in \Gamma} s_N(M_N-zI_N) \geq \gamma_\Gamma$. \end{proposition} 

\subsection{Ideas of the proof}
The proof of Proposition \ref{nonunif} is based on the two following key results.

\begin{proposition}\label{pasde} Assume that ${\bf (X1)}$  holds. 
Let K be a polynomial in  $u +t$ noncommutative variables. Define
$$K_N=K\left( \frac{X_N^{(1)}}{\sqrt{N}}, \ldots, \frac{X_N^{(u)}}{\sqrt{N}},  A_N^{(1)}, \ldots, A_N^{(t)}\right).$$
\begin{itemize}  
 
\item Assume  that \eqref{normedeA} holds.
Let $\{{a}_N^{(j)}, j=1,\dots,t\}$ be a set of noncommutative random variables in $\left({\cal A}, 
\phi\right)$ which is free from a free circular system $c=(c^{(1)},\dots,c^{(u)})$ in $\left({\cal A}, 
\phi\right)$ and such that  the $\ast$-distribution of $({A}_N^{(j)}\!,j\!=\!1,\dots,t)$ in the 
noncommutative probability space $\left({\cal M}_{N}(\C),\frac{1}{N} \Tr\right)$ coincides with the 
$\ast$-distribution of $a_N=({ a}_N^{(j)}, j=1,\ldots,t)$ in $\left({\cal A},\phi\right).$ Let $\tau_N$ be 
the the distribution of 
$$
K(c, a_N)\left[K( c, a_N)\right]^*
$$ 
with respect to $\phi$. If $[x,y]$, $x<y$, is such that there exists a $\delta>0$ such that for all large 
$N$, $(x-\delta; y+\delta)\subset\mathbb{R}\setminus\rm{supp}(\tau_N)$, then, we have
$$
\mathbb P[\mbox{for all large N},{\rm spect}(K_NK_N^*)\subset\mathbb{R}\setminus[x,y]]=1.
$$
\item Assume  that {\bf (A1)} holds.
Then, almost surely, the sequence of  $u+t$-tuples $\left( \frac{X_N^{(1)}}{\sqrt{N}}, \ldots, \frac{X_N^{(u)}}{\sqrt{N}},  A_N^{(1)}, \ldots, A_N^{(t)}\right)_{N\geq 1}$ converges in $\ast$-distribution towards $(c,a)$
 where $c=(c_1,\ldots, c_u)$ is a free circular system which is free with $a=(a^{(1)}, \ldots, a^{(t)})$ in $\left(  {\cal A},  \phi\right)$.
\end{itemize}
\end{proposition}
\begin{proposition}\label{transfert}
Consider a polynomial $P({Y}_1,{Y}_2)$, where ${Y}_1$ is a tuple of noncommuting 
nonselfadjoint indeterminates, ${Y}_2$ is a tuple of selfadjoint indeterminates, and no selfadjointness 
is assumed for $P$. We evaluate $P$ in $(c,a)$ and $(c,a_N)$, where $c$ is a tuple of free circulars, which 
is $*$-free from the tuples $a$ and $a_N$. We assume that $a_N\to a$ in moments and that there 
exists a $\tau>0$ such that $\sup_N\| a_N\| \leq \tau$.
\begin{enumerate}

\item We fix $z_0\in\mathbb C$ such that $|P(c,a)-z_0|^2\ge\delta_{z_0}>0$ for a fixed 
$\delta_{z_0}$.
\item We assume that there exists $N_{\delta_{z_0}}\in\mathbb N$ such that if $N\ge N_{\delta_{z_0}}$,
then $|P(0,a_N)-z_0|^2\ge\delta_{z_0}$.
\end{enumerate}
Then,  there exists $\epsilon_{z_0}>0$ for which there exists an $N_{\epsilon_{z_0}}\in
\mathbb N$ such that if $N\ge N_{\epsilon_{z_0}}$, then $|P(c,a_N)-z_0|^2\ge\epsilon_{z_0}$.
\end{proposition}

\begin{remark}
Of course Proposition \ref{transfert} still holds dealing with nonselfadjoint tuples $a_N$ by considering the selfadjoint tuples $(\Im (a_N), \Re (a_N))$.
\end{remark}

Let us explain  how to deduce Theorem \ref{outlier} from Proposition \ref{pasde}  and  Proposition \ref{transfert}.

\noindent Define $\mu_{N,z}$ as the distribution of $$\left[P( c^{(1)}, \ldots, c^{(u)},  a_N^{(1)}, \ldots, a_N^{(t)}) -z1\right]\times\left[P( c^{(1)}, \ldots, c^{(u)},  a_N^{(1)}, \ldots, a_N^{(t)}) -z1\right]^*$$ where
$ \{ c^{(1)},  (c^{(1)})^*\}, \ldots, \{c^{(u)},(c^{(u)})^*\},\{  a_N^{(1)}, \ldots, a_N^{(t)}\}$ are free sets of noncommutative random variables
 and the $\ast$-distribution of   $(a_N^{(1)}, \ldots, a_N^{(t)})$ in $({\cal A}, \phi)$ coincide with the  $\ast$-distribution of   $(A_N^{(1)}, \ldots, A_N^{(t)})$ in $({\cal M}_N(\mathbb{C}),\tr_N)$.
$\mu_{N,z}$ is the so-called deterministic equivalent measure of the empirical spectral measure of $(M_N-zI_N)(M_N-zI_N)^*$.\\
The following is a straightforward consequence of Proposition \ref{transfert}.

\begin{corollary}\label{decolle} 
Let $z \in \mathbb{C}$ be such that $ 0 \notin \text{supp}(\mu_z)$; assume that   there exists $\eta_z>0$ such that   for all $N$ large enough, 
there is no singular value of 
$$P( 0_N, \ldots,0_N,  (A_N^{(1)}), \ldots, (A_N^{(t)}))-zI_N$$ in $[0,\eta_z]$.
Then,
there exists $\epsilon_z>0$, such that, for all large $N$, $$[0, \epsilon_z] \subset \mathbb{R}\setminus \text{supp}(  \mu_{N,z}).$$
\end{corollary}

Then, we can deduce from Corollary \ref{decolle} and Proposition \ref{pasde} that there exists some 
$\gamma_z>0$ such that almost surely for all large $N$, there is no singular value of $M_N-zI_N$ in 
$[0,\gamma_z]$.\\ By a compacity argument and the fact that $z\mapsto s_N(M_N-zI)$ is 1-Lipschitz,
it readily follows that for any compact $\Gamma\subset\{z\colon0\notin\text{supp}(\mu_z)\}$, there 
exists some $\gamma_\Gamma >0$ such that almost surely for all large $N$,
\begin{equation}\label{inf}
\inf_{z\in \gamma} s_N(M_N-zI_N)\geq \gamma_\Gamma,
\end{equation}
leading to Proposition \ref{nonunif}.

\subsection{Proof of Proposition \ref{pasde} }

Note that 
$$
\begin{pmatrix}K_NK_N^*&0\\0&K_N^*K_N\end{pmatrix}=\begin{pmatrix}0&K_N\\K_N^*&0 \end{pmatrix}^2,
$$
so that the spectrum of $K_NK_N^*$ coincides with the spectrum of 
${\begin{pmatrix} 0 & K_N\\ K_N^* & 0 \end{pmatrix}}^2$. Now \begin{eqnarray}
\begin{pmatrix} 0 & K\\ K^* & 0 \end{pmatrix} &=&\sum_{i=1}^p \begin{pmatrix} 0 & b_i m_i\\ \bar b_i m_i^* & 0 \end{pmatrix} \nonumber\\
&=&  \sum_{i=1}^p b_i\begin{pmatrix} 0 &  m_i\\ m_i^* & 0 \end{pmatrix} \begin{pmatrix} 0 &0\\0 &1 \end{pmatrix} + \bar b_i \begin{pmatrix} 0 &0\\0 &1 \end{pmatrix} \begin{pmatrix} 0 &  m_i\\ m_i^* & 0 \end{pmatrix}\label{1}
\end{eqnarray}
where the $m_i$'s are monomials and the $b_i$'s are complex numbers. 
Define    
$Q_1= \begin{pmatrix}  I_N & 0\\ 0 & 0\end{pmatrix}$, $Q_2= \begin{pmatrix}  0 & 0\\ 0 & I_N\end{pmatrix}$ and $R= \begin{pmatrix}  0 & I_N\\ 0 & 0\end{pmatrix}$, $S=\begin{pmatrix}  0 & I_N\\ I_N & 0\end{pmatrix}$.
Note that 
$$\begin{pmatrix}  0 & \frac{X_N^{(i)}}{\sqrt{N}}\\ 0 & 0\end{pmatrix}= \sqrt{2}Q_1 \frac{{\cal W}^{(i)}}{\sqrt{2N}} Q_2  $$ where the ${\cal W}^{(i)}$'s, $i=1,\ldots, u$,  are $2N\times 2N$ independent standard  Wigner matrices.
Now, note that as noticed by \cite{BSS} for any monomial $x_1 \cdots x_k$, 
\begin{equation}\label{2bis}\begin{pmatrix}  0 & x_1 \cdots x_k \\ (x_1 \cdots x_k)^*& 0\end{pmatrix}=\Pi_{k-1}  \begin{pmatrix} 0 & x_k\\ x_k^* & 0\end{pmatrix} \Pi_{k-1}^*\end{equation}
where $$\Pi_{k-1}=
  \begin{pmatrix} 0 & x_1\\ I& 0\end{pmatrix}S \begin{pmatrix}  0 & x_2\\ I & 0\end{pmatrix} S\cdots  S  \begin{pmatrix} 0 & x_{k-1}\\ I& 0\end{pmatrix} S.$$
Indeed, this can be proved by induction noting that 
$$ \begin{pmatrix} 0 & x_1\\ I & 0\end{pmatrix}  S \begin{pmatrix} 0 & x_2\\ x_2^* & 0\end{pmatrix} S  \begin{pmatrix} 0 &I\\  x_1^* & 0\end{pmatrix} 
= \begin{pmatrix} 0 & x_1x_2\\ x_2^*x_1^* & 0\end{pmatrix}.$$
Note also that \begin{equation}\label{3} S \begin{pmatrix} 0 & I\\ x_1^* & 0\end{pmatrix}S= \begin{pmatrix} 0 & x_1^*\\ I & 0\end{pmatrix}.\end{equation}
Set  for  j=1,\ldots, t, ${\bf A}_N^{(j)}=  \begin{pmatrix} 0 & A_N^{(j)} \\ 0 & 0\end{pmatrix}$.
\\
From \eqref{1}, \eqref{2bis}, \eqref{3}, it readily follows that there exists a polynomial $\hat K$ such that 
$\begin{pmatrix} 0 & K_N\\ K_N^* & 0 \end{pmatrix}$ is equal to $$ \hat K \left(Q_1,Q_2, R, R^*,  {\bf A}_N^{(j)}, ({\bf A}_N^{(j)})^*,  j=1,\ldots, t, \frac{{\cal W}^{(i)}}{\sqrt{2N}}, i=1, \ldots, u\right).$$ 
Now, define  for $j=1,\ldots, t$, ${\bf a}_N^{(j)}=  \begin{pmatrix} 0 & a_N^{(j)} \\ 0 & 0\end{pmatrix}$,
$q_1= \begin{pmatrix}  1_{\mathcal A} & 0\\ 0 & 0\end{pmatrix}$, $q_2= \begin{pmatrix}  0 & 0\\ 0 & 1_{\mathcal A}\end{pmatrix}$ and $r= \begin{pmatrix}  0 & 1_{\mathcal A}\\ 0 & 0\end{pmatrix}$. 
Let $s_i^{(1)}$,  $s_i^{(2)}$,$i=1, \ldots,u$  be   semicircular variables such that
$ \{s_1^{(1)}\}, \ldots   \{s_u^{(1)}\}, \{s_1^{(2)}\}, \ldots   \{s_u^{(2)}\}, \{c_1,c_1^*\},\ldots, \{c_u,c_u^*\}, \{{ a}_N^{(j)}, j=1,\ldots,t\}$ are free. Define for $i=1,\ldots,u$, $${\bf s_i}=\frac{1}{\sqrt{2}}\begin{pmatrix} s_i^{(1)} & c^{(i)}\\(c^{(i)})^*  &  s_i^{(1)} \end{pmatrix}.$$ 
Similarly, \\

$\begin{pmatrix} 0 & K(c_1, \ldots, c_u, a_N^{(1)}, \ldots, a_N^{(t)})\\ \left[K( c_1, \ldots, c_u, a_N^{(1)}, \ldots, a_N^{(t)})\right]^* & 0 \end{pmatrix}$ $$= \hat K \left(q_1,q_2, r, r^*,  {\bf a}_N^{(j)}, ({\bf a}_N^{(j)})^*,  j=1,\ldots, t, {\bf s_i}, i=1, \ldots, u\right).$$ 
It readily follows that, the spectrum of $K_NK_N^*$ coincides with the spectrum of $ \hat K \left(Q_1,Q_2, R, R^*,  {\bf A}_N^{(j)}, ({\bf A}_N^{(j)})^*,  j=1,\ldots, t, \frac{{\cal W}^{(i)}}{\sqrt{2N}}, i=1, \ldots, u\right)^2$ and the spectrum of  of $K(c_1, \ldots, c_u, a_N^{(1)}, \ldots, a_N^{(t)}) \left[K( c_1, \ldots, c_u, a_N^{(1)}, \ldots, a_N^{(t)})\right]^*$ coincides with the spectrum of 
$\hat K \left(q_1,q_2, r, r^*,  {\bf a}_N^{(j)}, ({\bf a}_N^{(j)})^*,  j=1,\ldots, t, {\bf s_i}, i=1, \ldots, u\right)^2.$

Now, it is straightforward to see that the $\ast$-distribution of 
$(q_1,q_2,r, {\bf a}_N^{(j)}, j=1,\ldots,t) $ in $({\cal M}_2(\mathcal{A}),\tr_2\otimes\phi)$ coincides with the $\ast$-distribution of 
$(Q_1,Q_2,R,$ ${\bf A}_N^{(j)},j=1,\dots,t)$ in $({\cal M}_{2N}(\mathbb{C}),\tr_{2N})$. Moreover, by Lemma \ref{transfertliberte},  it turns out that the ${\bf s_i}$'s are free semicircular variables which are free with $(q_1,q_2,r,{\bf a}_N^{(j)}, j=1,\dots,t)$ 
in $({\cal M}_2(\mathcal{A}),\tr_2\otimes \phi)$.\\
Therefore, the first assertion of Proposition \ref{pasde} follows by applying  \cite[Theorem 1.1. and Remark 4]{BC}. 
The second assertion of Proposition \ref{pasde} can be proven by the same previous arguments.
Indeed, there exists a polynomial $\tilde K$ such that \\

\noindent $ \frac{1}{N} \Tr K\left( \frac{X_N^{(1)}}{\sqrt{N}}, \ldots, \frac{X_N^{(u)}}{\sqrt{N}},  A_N^{(1)}, \ldots, A_N^{(t)}\right)$
\begin{eqnarray*}
&=& \frac{1}{N} \Tr \left\{\begin{pmatrix}0&K_N\\K_N^*&0 \end{pmatrix} R^*\right\}\\&=& 2 \frac{1}{2N} \Tr  \tilde K \left(Q_1,Q_2, R, R^*,  {\bf A}_N^{(j)}, ({\bf A}_N^{(j)})^*,  j=1,\ldots, t, \frac{{\cal W}^{(i)}}{\sqrt{2N}}, i=1, \ldots, u\right) \end{eqnarray*}
Thus, using \cite[Proposition 2.2. and Remark 4]{BC}, we obtain that \\

$ \frac{1}{N} \Tr K\left( \frac{X_N^{(1)}}{\sqrt{N}}, \ldots, \frac{X_N^{(u)}}{\sqrt{N}},  A_N^{(1)}, \ldots, A_N^{(t)}\right)$
$$\underset{N\rightarrow +\infty}{\rightarrow} 2 \tr_2\otimes \phi \left[ \tilde  K \left(q_1,q_2, r, r^*,  {\bf a}^{(j)}, ({\bf a}^{(j)})^*,  j=1,\ldots, t, {\bf s_i}, i=1, \ldots, u\right)\right]$$ where, for $j=1,\ldots, t$, ${\bf a}^{(j)}=  \begin{pmatrix} 0 & a^{(j)} \\ 0 & 0\end{pmatrix}$. Now, \\

$  2 \tr_2\otimes \phi \left[\tilde  K \left(q_1,q_2, r, r^*,  {\bf a}^{(j)}, ({\bf a}^{(j)})^*,  j=1,\ldots, t, {\bf s_i}, i=1, \ldots, u\right)\right]$ \begin{eqnarray*} &=& 2 \tr_2\otimes \phi   \left\{\begin{pmatrix}0&K(c,a)\\K(c,a)^*&0 \end{pmatrix} r^*\right\}\\&=& \phi \left(K(c,a)\right).\end{eqnarray*} 
The second assertion of Proposition \ref{pasde} follows.

\subsection{Proof of Proposition \ref{transfert}}\label{serbanlinearisation}

We prove this using linearization and hermitization. Our linearization of a nonselfadjoint polynomial 
will naturally not be selfadjoint, so the results from \cite{BBC} do not apply directly to it, but some of
the methods will. Before we analyze this linearization, let us lay down the steps that we shall take
in order to prove the above result. Let $L$ be our linearization of $P({Y}_1,{Y}_2)-z_0$.
\begin{enumerate}

\item We have $|P(c,a_N)\!-\!z_0|^2\ge\epsilon_{z_0}\!\iff\!\begin{bmatrix}
0 & P(c,a_N)\!-\!z_0\\
(P(c,a_N)\!-\!z_0)^* & 0
\end{bmatrix}^2\!\ge\epsilon_{z_0}.$
\item 
There exists $\iota=\iota({\epsilon_{z_0}},P,\tau)>0$ such that 
$$
\!\!\left|\begin{bmatrix}
0 & P(c,a_N)\!-\!z_0\\
(P(c,a_N)\!-\!z_0)^* & 0
\end{bmatrix}\right|\ge\epsilon_{z_0}\!
\iff\! \left|\begin{bmatrix}
0 & L(c,a_N)\\
L(c,a_N)^* & 0
\end{bmatrix}\right|\ge\iota.
$$
\item We write 
$$
\begin{bmatrix}
0 & L(c,a_N)\\
L(c,a_N)^* & 0
\end{bmatrix}=\begin{bmatrix}
0 & L(0,a_N)\\
L(0,a_N)^* & 0
\end{bmatrix}+C,$$
where $C$ is a selfadjoint matrix containing only circular variables and their adjoints. It will be clear
that $
\begin{bmatrix}
0 & L(c,a_N)\\
L(c,a_N)^* & 0
\end{bmatrix}$ contains at most one nonzero element per row/column, except possibly for the first
row/column.
\item We use Lemma \ref{symm} to conclude that the lhs of the previous item is invertible if and only if
$$
\begin{bmatrix}
0 & (I_m\otimes \epsilon) L(0,a_N)\\
(I_m\otimes \epsilon) L(0,a_N)^* & 0
\end{bmatrix}+\mathcal S$$
is, where $\mathcal S$ is obtained from $C$ by replacing each circular entry with a semicircular from the 
same algebra (and hence free from $a_N$), and $\epsilon$ is a $\{-1,1\}$-Bernoulli distributed random 
variable which is independent from $a_N$ and free from $\mathcal S$. As noted in Example \ref{sem}, since 
$C=C^*$, $\mathcal S$ is indeed a matrix-valued semicircular random variable.
\item We apply Lemma \ref{Serban} to the above item in order to determine under what conditions the
sum in question has a spectrum uniformly bounded away from zero.
\item Finally, we use the convergence in moments of $a_N$ to $a$ in order to conclude that the 
conditions obtained in the previous item are satisfied by $\begin{bmatrix}
0 & (I_m\otimes \epsilon) L(0,a_N)\\
(I_m\otimes \epsilon) L(0,a_N)^* & 0
\end{bmatrix}+\mathcal S$.
\end{enumerate}

Part (1) is trivial:
$$
\begin{bmatrix}
0 & P(c,a_N)-z_0\\
(P(c,a_N)-z_0)^* & 0
\end{bmatrix}^2\!=\begin{bmatrix}
|P(c,a_N)-z_0|^2 & 0\\
0  & |(P(c,a_N)-z_0)^*|^2 
\end{bmatrix}.
$$
Since our variables live in a II${}_1$-factor, the two nonzero entries of the right hand side have the same
spectrum.

Part (2) requires a careful analysis of the linearization we use. The construction from \cite{A} proceeds by 
induction on the number of monomials in the given polynomial. 
If $P=X_{i_1}X_{i_2}X_{i_3}\cdots X_{i_{\ell-1}}X_{i_\ell}$, where
$\ell\ge2$ and $i_1,\dots,i_\ell\in\{1,\dots,k\}$, we set $n=\ell$ and
$$
L=-\begin{bmatrix}
0 & 0 & \cdots &  0 & X_{i_1}\\
0 & 0 & \cdots &  X_{i_2} & -1\\
\vdots&\vdots& \reflectbox{$\ddots$}&\vdots&\vdots\\
0 & X_{i_{\ell-1}}&\cdots&0&0\\
X_{i_\ell}&-1&\cdots&0&0
\end{bmatrix}.
$$
However, unlike in \cite{A,BBC}, we choose here $L$ to be
$$
L=-\begin{bmatrix}
0 & 0 & \cdots & 0 & 0 & 1\\
0 & 0 & \cdots &  0 & X_{i_1} & -1\\
0 & 0 & \cdots &  X_{i_2} & -1 & 0 \\
\vdots&\vdots& \reflectbox{$\ddots$}&\vdots&\vdots&\vdots\\
0 & X_{i_{\ell}}&\cdots&0&0 & 0\\
1 &-1&\cdots&0&0& 0
\end{bmatrix}.
$$
That is, we apply the procedure from \cite{A}, but to $P=1X_{i_1}X_{i_2}X_{i_3}\cdots X_{i_{\ell-1}}
X_{i_\ell}1$. If $\ell=1$, we simply complete $X$ to $1X1$. Even if we have a multiple of $1$, we 
choose here to proceed the same way. The lower right $\ell\times\ell$ corner of this matrix has an 
inverse of degree $\ell$ in the algebra ${\cal M}_{\ell}(\mathbb C\langle X_1,\dots,X_k\rangle)$. (The constant 
term in this inverse is a selfadjoint matrix and its spectrum is contained in $\{-1,1\}$.) The first row
contains only zeros and ones, and the first column is the transpose of the first row. Suppose now that 
$p=P_1+P_2$, where $P_1,P_2\in \mathbb C\langle X_1,\dots,X_k\rangle$, and that linear polynomials
$$
L_{j}=\begin{bmatrix}
0 & u_j^*\\
u_j & Q_j
\end{bmatrix}\in {\cal M}_{n_j}(\mathbb C\langle X_1,\dots,X_k\rangle),\quad j=1,2,
$$
linearize $P_1$ and $P_2$.  Then we set $n=n_1 +n_2 -1$ and observe that the matrix 
$$L=
\begin{bmatrix}
0 & u_1^* & u_2^*\\
u_1 & Q_1 & 0\\
u_2 & 0 &Q_2
\end{bmatrix}=
\begin{bmatrix}
0 & u^*\\
u & Q
\end{bmatrix}
\in {\cal M}_{n_1+n_2-1}(\mathbb C\langle X_1,\dots X_k\rangle).
$$
is a linearization of $P_1+P_2$. $L$ is built so that $\left[(ze_{1,1}-L)^{-1}\right]_{1,1}=(z-P)^{-1}$, 
$z-P$ is invertible if and only if $(ze_{1,1}-L)$ is invertible, and each row/column of the matrix $L$, 
except possibly for the first, contains at most one nonzero indeterminate (i.e. non-scalar).
By applying the linearization process to $1X_{i_1}X_{i_2}X_{i_3}\cdots X_{i_{\ell-1}}X_{i_\ell}1$
instead of $X_{i_1}X_{i_2}X_{i_3}\cdots X_{i_{\ell-1}}X_{i_\ell}$, we have insured that there is at most one 
nonzero indeterminate in each row/column. An important side benefit is that with this modification, we 
may assume that, with the notations from item \ref{linny} of Section \ref{Sec:linny},
$$
v=u, \text{ and all entries of this vector are either 0 or 1.}
$$
While this linearization is far from being minimal, and should not be used for practical computations, it 
turns out to simplify to some extent the notations and arguments of our proofs.

The concrete expression of the inverse of $ze_{1,1}-L$ in terms of $
L=\begin{bmatrix}
0 & u^*\\
u & Q
\end{bmatrix}$ is provided by the Schur complement formula as
$$
(ze_{1,1}-L)^{-1}=\begin{bmatrix}
(z-u^*Q^{-1}u)^{-1} & -(z-u^*Q^{-1}u)^{-1}u^*Q^{-1}\\
-Q^{-1}u(z-u^*Q^{-1}u)^{-1} & Q^{-1}+Q^{-1}u(z-u^*Q^{-1}u)^{-1}u^*Q^{-1}
\end{bmatrix}.
$$
It follows easily from this formula that $z-P$ is invertible if and only if $ze_{1,1}-L$ is invertible. It was
established in \cite[Lemma 4.1]{BBC} that $Q$, and hence $Q^{-1}$, is of the form $T(1+N)$ for some 
permutation scalar matrix $T$ and nilpotent matrix $N$, which may contain non-scalar entries. Let us 
establish a non-selfadjoint (and thus necessarily weaker) version of \cite[Lemma 4.3]{BBC}.

\begin{lemma}\label{esty}
Assume that $P\in\mathbb C\langle{Y_1},{ Y_2}\rangle$ is an arbitrary polynomial in the 
non-selfadjoint indeterminates ${Y_1}$ and selfadjoint indeterminates ${Y_2}$. Let $L$ be 
a linearization of $P$ constructed as above. Given tuples of noncommutative random variables
$c$ and $a$, for all $\delta>0$ such that $|P(c,a)|^2>\delta$, there exists $\mathfrak e>0$ such that $|L(c,a)|^2>\mathfrak e$, and the number $\mathfrak e$ only depends on $\delta>0,$ $P,$ and the 
supremum of the norms of $c,a$. Conversely, for all $\mathfrak 
e>0$ such that $|L(c,a)|^2>\mathfrak 
e$, there exists $q>0$ such that $|P(c,a)|^2>q>0$ and $q$ depends only on $\mathfrak e$, $P,$ and the supremum of 
the norms of $c,a$.
\end{lemma}

\begin{proof}
With the decomposition $L\!=\!\begin{bmatrix}
0 & u^*\\
u & Q
\end{bmatrix}$, we have $|L|^2\!=\!\begin{bmatrix}
u^*u & u^*Q^*\\
Qu & \!uu^*\!+\!QQ^*
\end{bmatrix}$. Recall that $|P|^2=u^*Q^{-1}uu^*(Q^{-1})^*u$. Now consider these expressions evaluated
in the tuples of operators mentioned in the statement of the lemma. In order to save space, we will 
nevertheless suppress them from the notation. We assume that $|P|^2>\delta$. Strangely enough, it
will be more convenient to estimate an upper bound for $|L|^{-2}$ rather than a lower bound for $|L|^2$.
The entries of $|L|^{-2}$ expressed in terms of the above decomposition are
\begin{eqnarray*}
(|L|^{-2})_{1,1}& = & \left(u^*u-u^*Q^*(uu^*+QQ^*)^{-1}Qu\right)^{-1}, \\
(|L|^{-2})_{1,2}& = & -\left(u^*u-u^*Q^*(uu^*+QQ^*)^{-1}Qu\right)^{-1}u^*Q^*(uu^*+QQ^*)^{-1},\\
(|L|^{-2})_{2,1}& = & -(uu^*+QQ^*)^{-1}Qu\left(u^*u-u^*Q^*(uu^*+QQ^*)^{-1}Qu\right)^{-1}, \\
(|L|^{-2})_{2,2}& = & 
(uu^*\!\!+\!QQ^*)^{-1}\!Qu\!\left(\!u^*\!u\!-\!u^*Q^*(uu^*\!\!+\!QQ^*)^{-1}\!Qu\!\right)^{-1}\!\!
u^*Q^*\!(uu^*\!\!+\!QQ^*)^{-1}\\
& & \mbox{}+(uu^*+QQ^*)^{-1}.
\end{eqnarray*}
We only need to estimate the norms of the above elements in terms of $\delta$, $P$, and the norms
of the variables in which we have evaluated the above. It is clear that 
\begin{eqnarray*}
(|L|^{-2})_{1,1}& = & 
\left(u^*Q^{-1}u(1+u^*(Q^*)^{-1}Q^{-1}u)^{-1}u^*(Q^*)^{-1}u\right)^{-1}\\
& = & \left(P(1+u^*(Q^*)^{-1}Q^{-1}u)^{-1}P^*\right)^{-1}\\
& \leq & \left(P(\|1+u^*(Q^*)^{-1}Q^{-1}u\|)^{-1}P^*\right)^{-1}\\
& = & (1+\|u^*(Q^*)^{-1}Q^{-1}u\|)|P|^{-2}.
\end{eqnarray*}
Similarly, $(uu^*+QQ^*)^{-1}\leq(QQ^*)^{-1}\leq\|Q^{-1}\|^2.$ We obtain this way the following 
majorizations for each of the entries, which will allow us to estimate $\mathfrak e$ (these majorizations 
are not optimal, but close to):
\begin{eqnarray*}
\|(|L|^{-2})_{1,1}\|& \le &  (1+\|u^*(Q^*)^{-1}Q^{-1}u\|)\||P|^{-2}\|, \\
\|(|L|^{-2})_{1,2}\|& \le & (1+\|u^*(Q^*)^{-1}Q^{-1}u\|)\||P|^{-2}\|\|u^*\|\|Q^*\|\|Q^{-1}\|^2,\\
\|(|L|^{-2})_{2,1}\|& \le & \|Q^{-1}\|^2\|Q\|\|u\|(1+\|u^*(Q^*)^{-1}Q^{-1}u\|)\||P|^{-2}\|, \\
\|(|L|^{-2})_{2,2}\|& \le & \|Q^{-1}\|^4\|Q\|^2\|u\|^2(1+\|u^*(Q^*)^{-1}Q^{-1}u\|)\||P|^{-2}\|+\|Q^{-1}\|^2.
\end{eqnarray*}
We shall not be much more explicit than this, but let us nevertheless comment on why the above satisfies
the corresponding conclusion of our lemma. As noted before, $u$ is a vector of zeros and ones. It 
follows immediately from the construction of $L$ that the number of ones is dominated by the number of 
monomials of $P$, quantity clearly depending only on $P$. Recall that $Q$ is of the form $T(1+N)$, with 
$T$ a permutation matrix, and $N$ a nilpotent matrix. The norm of $T$ is necessarily one. The nilpotent 
matrix corresponding to $Q$ is simply a block upper diagonal matrix (i.e. a matrix which has on its 
diagonal a succession of blocks, each block being itself an upper diagonal matrix) with entries which are 
operators from the tuples $a$ and $c$ in which we evaluate $P$ (and $L$). Its norm is trivially bounded
by the supremum of all the norms of the operators involved times the supremum of all the scalar 
coefficients. Since $\|Q^{-1}\|=\|T^{-1}(1+N)^{-1}\|\le1+\sum_{j=1}^m\|N\|^j$, where $m$ is the size
of the linearization, we obtain an estimate for $\|Q^{-1}\|$ from above by $(m+1)(1+\|Q\|)^m$. Finally,
$\||P|^{-2}\|\le\delta^{-1}$. This guarantees that $\||L|^{-2}\|$ is bounded from above, so that $|L|^2$ 
is bounded from below, by a number $\mathfrak e$ depending on $\delta$, $P$, and the norms of the entries of $P$.

Conversely, assume that $|L|^2>\mathfrak e$ for a given strictly positive constant $\mathfrak e$. As 
before, this is equivalent to $\||L|^{-2}\|<\frac{1}{\mathfrak e}$, which allows for the estimate of the 
$(1,1)$ entry of $|L|^{-2}$ by $\left\|\left(P(1+u^*(Q^*)^{-1}Q^{-1}u)^{-1}P^*\right)^{-1}\right\|<
\frac1{\mathfrak e}$, so that
$$
\left(P(1+u^*(Q^*)^{-1}Q^{-1}u)^{-1}P^*\right)^{-1}<\frac1{\mathfrak e},
$$
as inequality of operators. This tells us that $P(1+u^*(Q^*)^{-1}Q^{-1}u)^{-1}P^*>\mathfrak e$, so that
$$
PP^*>\frac{\mathfrak e}{\|(1+u^*(Q^*)^{-1}Q^{-1}u)^{-1}\|}\geq\mathfrak e.
$$
This concludes the proof.
\end{proof}

Part (3) is a simple formal step.

Step (4) becomes a direct consequence of Lemma \ref{symm}. 

Now, in step (5),  we finally involve our variables $c,a,a_N$ directly. We have assumed that
$|P(c,a)-z_0|^2>\delta_{z_0}>0$, so that, according to steps (1) and (2), we have 
$\left|\begin{bmatrix} 0 & L(c,a)\\ L(c,a)^* & 0 \end{bmatrix}\right|>\iota$ for a $\iota>0$ depending,
according to step (2), only on $P$, $\delta_{z_0}$, and the norms of $c,a$. According to step (4),
it follows that $\begin{bmatrix} 0 & (I_m\otimes \epsilon) L(0,a)\\ (I_m\otimes \epsilon)L(0,a)^* & 0 \end{bmatrix}+\mathcal S$
is invertible; moreover, the norm of the inverse is bounded in terms of $P$, $\delta_{z_0}$, and the norms 
of $c,a$.
According to Lemma \ref{Serban} and Remark \ref{meh}, denoting $\mathcal Y=
\begin{bmatrix} 0 & (I_m\otimes \epsilon) L(0,a)\\ (I_m\otimes \epsilon) L(0,a)^* & 0 \end{bmatrix}$, the condition of invertibility
of $\mathcal{S+Y}$ is equivalent to the invertibility of $\mathcal Y$ together with the existence of an
$r\in(0,1)$ such that 
$\text{spect}(\eta\circ G_\mathcal Y'(0))\subset(1-r)\mathbb D$. We naturally denote $\mathcal Y_N=
\begin{bmatrix} 0 & (I_m\otimes \epsilon) L(0,a_N)\\ (I_m\otimes \epsilon) L(0,a_N)^* & 0 \end{bmatrix}$. We have
assumed that $|P(0,a_N)-z_0|^2>\delta_{z_0}$ for all (sufficiently large) $N\in\mathbb N$, so that
$|\mathcal Y_N|^2>\zeta$ for a $\zeta$ that only depends on $P,\delta_{z_0}$, and the supremum
of the norms of $a_N$, which is assumed to be bounded. Thus, $|\mathcal Y_N|^2$ is uniformly
bounded from below as $N\to\infty$. In order to insure the invertibility of $\mathcal{S+Y}_N$,
we also need that $\text{spect}(\eta\circ G_{\mathcal Y_N}'(0))\subset\overline{\mathbb D}\setminus\{1\}$, 
for all $N$ sufficiently large. The existence of $G_{\mathcal Y_N}'(0)$ is guaranteed by the hypothesis of 
invertibility of $\mathcal Y_N$. Since 
$$
G_{\mathcal Y_N}'(0)(c)=(id_m\otimes\varphi)\left[\mathcal Y_N^{-1}c\mathcal Y_N^{-1}\right],
$$
and 
$$
\mathcal Y_N^{-1}=\begin{bmatrix}0&(I_m\otimes \epsilon)(L(0,a_N)^*)^{-1}\\(I_m\otimes \epsilon) L(0,a_N)^{-1}&0 \end{bmatrix},
$$
we only need to remember that all entries of $L(0,a_N)^{-1}$ are products of polynomials in $a_N$
and $(P(0,a_N)-z_0)^{-1}$ in order to conclude that the convergence in moments of $a_N$ to $a$
implies the convergence in norm 
of $G_{\mathcal Y_N}'(0)$ to $G_{\mathcal Y}'(0)$ (recall that, according to hypothesis 2. in the statement 
of our proposition, $|P(0,a_N)-z_0|^2>\delta_{z_0}>0$ uniformly). Thus, for $N$ sufficiently large, all 
eigenvalues of $\eta\circ G_{\mathcal Y_N}'(0)$ are included in $(1-\frac{r}{2})\mathbb D$. 
This guarantees the invertibility of all $\mathcal{S+Y}_N$ for $N$ sufficiently large.

To prove item (6) and conclude our proof, we only need to show that for $N$ sufficiently large,
$|\mathcal{S+Y}_N|^2>\frac{\iota}{2}$. There is a simple abstract shortcut for this: as Proposition \ref{2} 
shows, the endpoint of the support of the (scalar) distribution of $\mathcal{S+Y}_N$
is given by that first $x_N\in(0,+\infty)$ for which $1\in\text{spect}(\eta\circ G_{\mathcal{Y}_N}'(x_N))$.
On the one hand, $G_{\mathcal{Y}_N}$ is guaranteed to be analytic on $[0,\delta_{z_0}]$. On the other,
since ${\mathcal{Y}_N}\to\mathcal Y$ in distribution, we have $G_{\mathcal{Y}_N}\to G_{\mathcal{Y}}$
uniformly on $[0,\delta_{z_0}-\varepsilon]$ for any fixed $\varepsilon>0$. In particular, 
$G_{\mathcal{Y}_N}'(x)\to G_{\mathcal{Y}}'(x)$ for any $x$ in this interval. Thus, $x_N$ is 
bounded away from zero uniformly in $N$ as $N\to\infty$. A second application of the convergence
of $G_{\mathcal{Y}_N}$ allows us to conclude.

\section{Stable outliers; proof of Theorem \ref{outlier}}\label{sp}
Making use of a linearization procedure, the  proof closely follows the approach of \cite{BoC}. The most significant  novelty is Proposition \ref{conjecture} which substantially generalizes Theorem 1.3. A. in \cite{CI} (see also Proposition 2.1 in \cite{BoC}) and whose proof relies on operator-valued free probability results established in Section \ref{prelim}. Nevertheless, we precise all arguments for the reader's convenience.

Let 
$$
L_P=\gamma\otimes1+\sum_{j=1}^{u}\zeta_j\otimes y_{j}+\sum_{k=1}^{t}\beta_k\otimes y_{u+k},
$$ 
be a linearization of $P(y_1,\ldots, y_{u+t})$ with coefficients in ${\cal M}_m(\mathbb{C})$ such that, for any $u+t$-tuple $y$ of matrices, $|\det Q(y)|=1$ 
 (see \eqref{detQ}).\\ Let $\Gamma$ be a compact set in $\mathbb{C}\setminus\text{spect}
(P(c,a))$. Note that 
$$
\min_{z\in\partial\Gamma}\left|\frac{\det (z I_N-P(0_N,\dots,0_N,A_N^{(1)},\dots,A_N^{(t)}))}{\det(zI_N-P(0_N,\dots,0_N,
(A_N^{(1)})^{'} ,\ldots, (A_N^{(t)})^{'})}\right|\geq \epsilon
$$ 
is equivalent to 
 \begin{eqnarray}\label{conditionRouche2}
\lefteqn{
\min_{z\in\partial\Gamma}\left|\frac{\det(zI_N-P(0_N,\dots,0_N,A_N^{(1)},\dots,A_N^{(t)}))}{\det(zI_N-P(0_N,\ldots,0_N,
(A_N^{(1)})^{'} , \ldots, (A_N^{(t)})^{'})}\right.}\\
& & \quad\quad\quad\quad\quad\quad\quad\quad\quad\quad\quad\quad
\mbox{}\times\left.\frac{\det(Q(0_N,\dots,0_N,A_N^{(1)},\ldots,A_N^{(t)}))}{\det Q(0_N,\dots,0_N,(A_N^{(1)})^{'},\dots,
(A_N^{(t)})^{'})}\right| \geq \epsilon,\nonumber
\end{eqnarray} 
since $|\det Q|$ is constant. Now, following the proof of Lemma 4.3 in \cite{BBC}, one can see that this is  also equivalent to 
\begin{equation}\label{conditionRouche}
\min_{z\in\partial\Gamma}\left|\frac{\det(ze_{11}\otimes I_N-\gamma\otimes I_N-\sum_{k=1}^t\beta_k 
\otimes A_N^{(k)})}{\det(ze_{11}\otimes I_N-\gamma\otimes I_N-\sum_{k=1}^t\beta_k\otimes(A_N^{(k)})^{'}  )}
\right|\geq\epsilon.
\end{equation}
According to Lemma \ref{inversible}, the eigenvalues of ${ M}_N$ are the zeroes of $z\mapsto\det(ze_{11} 
\otimes I_N -\gamma\otimes I_N-\sum_{j=1}^u\zeta_j\otimes \frac{X_N^{(j)}}{\sqrt{N}}-\sum_{k=1}^t\beta_k\otimes
(A_N^{(k)}))$. By Assumption (${\bf A_2'}$), Proposition \ref{nonunif} and Lemma \ref{inversible}, almost surely for all large $N$,  for any $z\in\Gamma$, we 
can define
$$
R_N(z)=(ze_{11}\otimes I_N-\gamma\otimes I_N-\sum_{j=1}^u\zeta_j\otimes \frac{X_N^{(j)}}{\sqrt{N}} -
\sum_{k=1}^t\beta_k \otimes(A_N^{(k)})^{'} )^{-1},
$$
$$
R{'}_N(z)=(ze_{11} \otimes I_N -\gamma \otimes I_N  -\sum_{k=1}^t\beta_k \otimes
(A_N^{(k)})^{'}  )^{-1}.
$$
Note that, since each $(A_N^{(k)})^{''}$ has a bounded rank  $r_k(N)=O(1)$,  there exist matrices $P_N\in {\cal M}_{mN, p}$, $Q_N \in M_{p, mN}$, where $p$ is fixed, such that
\begin{equation}\label{defPQ}
\sum_{k=1}^t\beta_k \otimes(A_N^{(k)})^{''} =P_NQ_N.
\end{equation}
Recall Sylvester's identity: if $P, Q^{\top}\in {\cal M}_{d_1,d_2}(\mathbb{C})$, 
$$
\det ( I_{d_1} + PQ ) = \det ( I_{d_2} + QP ).
$$
Using this identity, it is clear that, almost surely for all large $N$,  the eigenvalues of ${ M}_N$ in $\Gamma$ are precisely the zeros of the 
random analytic function $z \mapsto \det (I_p-Q_N R_N(z) P_N)$ in that set. \\
Now, similarly,  for any $z$ in $\Gamma$,
\begin{equation} \label{quotient}
\det(I_p-Q_NR{'}_N(z)P_N)=\frac{\det(ze_{11}\otimes I_N -\gamma\otimes I_N -\sum_{k=1}^t\beta_k\otimes
A_N^{(k)})}{\det(ze_{11}\otimes I_N -\gamma \otimes I_N  -\sum_{k=1}^t\beta_k \otimes(A_N^{(k)})^{'})}.
\end{equation}
Thus, the zeroes of $z \mapsto \det(I_p-Q_N R{'}_N(z)P_N)$ in $\Gamma$ are the zeroes of $z \mapsto \det(ze_{11} \otimes 
I_N-\gamma\otimes I_N-\sum_k\beta_k\otimes A_N^{(k)})$ in $\Gamma$, that is, the eigenvalues in 
$\Gamma$ of  
$$
M^{(0)}_N=P( 0_N, \ldots, 0_N, A_N^{(1)}, \ldots, A_N^{(t)}).
$$
The rest of the proof is devoted to establish that $\det(I_p-Q_NR_N(z)P_N)-\det(I_p-Q_NR{'}_N(z)P_N)$ converges uniformly in $\Gamma$ to zero.\\

\noindent {\bf Step 1: Iterated resolvent identities.}\\
  
\noindent Set 
$$
Y_N=\sum_{j=1}^u \zeta_j \otimes \frac{X_N^{(j)}}{\sqrt{N}}.
$$
Using repeatedly the resolvent identity,
$$
R_N(z)= R'_{N}(z) +  R'_N(z)  Y_N R_N(z),
$$
we find that, for any integer number $K \geq 2$, 
\begin{align} & Q_NR_N(z) P_N  -  Q_N R'_N(z) P_N  \nonumber\\
&\quad\quad=\sum_{k=1}^{K-1}Q_N\left(R'_N(z)Y_N\right)^kR'_N(z)P_N+Q_N\left(R'_N(z)Y_N\right)^{K}
R_N(z)P_N.\label{exp}
\end{align}
The following two steps will be of basic use to prove the uniform convergence in $\Gamma$ of  the right hand side of \eqref{exp} towards zero.\\

\noindent {\bf Step 2: Study of the spectral radius of $R'_N(z)Y_N$.}\\
\noindent The aim of this second step is to prove Lemma \ref{rayonspectral} which establishes an upper bound strictly smaller than 1 of the spectral norm of $R'_N(z)Y_N$. The proof of Lemma \ref{rayonspectral} is based on Proposition \ref{nonunif} and the characterization, provided by Lemma \ref{Serban}, of the invertibility of the sum of a centered ${\cal M}_m(\C)$-valued semi-circular $s$ and some selfadjoint $y\in {\cal M}_m(\mathcal{A})$ with non-commutative symetric entries such that $s$ and $y$ are free over ${\cal M}_m(\C)$.
Recall that $\mu_z$ is the distribution of 
$$
\left[P(c^{(1)},\dots,c^{(u)},a^{(1)},\dots,a^{(t)})-z1_{\mathcal{A}}\right]\left[P(c^{(1)},\dots,c^{(u)},a^{(1)},\dots,a^{(t)})-
z1_{\mathcal{A}}\right]^*.
$$
Define $\nu_z$ as the distribution of  
$$
\left[P( 0_{\mathcal{A}}, \ldots, 0_{\mathcal{A}}, a^{(1)},\dots,a^{(t)})-z1_{\mathcal{A}}\right] \left[P( 0_{\mathcal{A}}, \ldots, 0_{\mathcal{A}}, a^{(1)},\dots,a^{(t)})-z1_{\mathcal{A}}\right]^*, 
$$
and $S_0=\left\{z \in \mathbb{C}, 0\in  \text{supp}(\nu_z)\right\} .$

\begin{proposition}\label{conjecture}
Let 
$$
L_P=\gamma\otimes1+\sum_{j=1}^{u}\zeta_j\otimes y_{j}+\sum_{k=1}^{t}\beta_k\otimes y_{u+k},
$$ 
be a linearization of $P(y_1,\ldots, y_{u+t})$ with coefficients in ${\cal M}_m(\mathbb{C})$. Set 
$$
y_z=\sum_{k=1}^{t}\beta_k \otimes a^{(k)}+(\gamma-ze_{11})\otimes 1_{\mathcal{A}}.
$$ 
Let $\epsilon$ be some selfadjoint $\{-1,+1\}$-Bernoulli variable in ${\cal A}$ independent from the 
entries of $y_z$. Let $s_1,\ldots,s_u$ be free semicircular variables in ${\cal A}$ free from $\epsilon$ and 
the entries of $y_z$. Define 
$$
\mathcal Y_z=\begin{pmatrix} 0& (I_m\otimes \epsilon) y_z\\(I_m\otimes \epsilon) y_z^*& 0\end{pmatrix}\text{and~}\mathcal S=
\begin{pmatrix} 0&\sum_{j=1}^u\zeta_j\otimes s_j\\ \sum_{j=1}^u \zeta_j^*\otimes s_j& 0\end{pmatrix}.
$$
If $z\notin S_0$, let $\Delta_1(z)$ be the operator 
$$
\begin{array}{ll} 
{\cal M}_{2m}(\mathbb{C})\to {\cal M}_{2m}(\mathbb{C})\\
b\mapsto id_{2m}\otimes\phi\left(\mathcal S( [id_{2m}\otimes\phi((\mathcal Y_z)^{-1}(b\otimes 1)(
\mathcal Y_z)^{-1})]\otimes1)\mathcal S\right).
\end{array}
$$
We have $0\notin\text{supp}(\mu_z)$ iff $z\notin S_0$ and $\text{spect}(\Delta_1(z))\subseteq 
\overline{\mathbb D}\setminus\{1\}$.
\end{proposition}

\begin{proof}

According to Remark \ref{supportloi}, we have that $0\notin \text{supp}(\mu_{z})$ if and only if $P(c^{(1)},
\dots,c^{(u)},a^{(1)},\dots,a^{(t)})-z1$ is invertible. According to Lemma \ref{inversible}, it follows that 
$0\notin\text{supp}(\mu_{z})$ if and only if $\sum_{j=1}^u\zeta_j\otimes c^{(j)}+y_z$ is invertible.
 Now, $\sum_{j=1}^u \zeta_j \otimes c^{(j)} +y_z$ is invertible if and only if $\left[\sum_{j=1}^u \zeta_j\!
\otimes\!c^{(j)}\!+\!y_z\right]\!\!\left[\sum_{j=1}^u\zeta_j\!\otimes\!c^{(j)}\!+y_z\right]^*$ and 
$\left[\sum_{j=1}^u \zeta_j\!\otimes\!c^{(j)}\!+y_z\right]^*\!\left[\sum_{j=1}^u \zeta_j\!\otimes\!c^{(j)}\! 
+y_z\right]$ are invertible, and then, by Lemma \ref{symm}, since $\tr_m\!\otimes\phi$ is faithful, if and 
only if $\left[\sum_{j=1}^u\zeta_j\!\otimes\!s_j+(I_m\otimes \epsilon) y_z\right]\!\left[\sum_{j=1}^u\zeta_j\!\otimes
\!s_j+(I_m\otimes \epsilon) y_z\right]^*$ and $\left[\sum_{j=1}^u \zeta_j \otimes s_j +(I_m\otimes \epsilon) y_z\right]^*
\left[\sum_{j=1}^u \zeta_j \otimes s_j + (I_m\otimes \epsilon) y_z\right]$ are  invertible, that is if and only if 
${\cal S}+{\cal Y}_z$ is invertible. Thus, Proposition \ref{conjecture} follows from  Example \ref{sem} and  Lemma \ref{Serban}.
\end{proof}

\noindent  Define for any $w, z$ in $ \mathbb{C}$,
$\mu_{w,z}$ as the distribution of 
$$ \left[P( w c^{(1)}, \ldots, w c^{(u)},  a^{(1)},\dots,a^{(t)})-zI\right] \left[P( wc^{(1)}, \ldots, wc^{(u)}, a^{(1)},\dots,a^{(t)})-zI\right]^*.$$
\begin{lemma}\label{w}
$0\notin\text{supp}(\mu_{w,z})$ if and only if $z\notin S_0$ and $\text{spect}(|w|^2\Delta_1(z))\subseteq 
\overline{\mathbb D}\setminus\{1\}$, where $S_0$ and $\Delta_1(z)$ are defined in Proposition \ref{conjecture}.
\end{lemma}
\begin{proof} 
Note that $(c^{(1)},\ldots,c^{(u)}) $ and  $(\exp(i \arg w)c^{(1)},\ldots,\exp(i \arg w)c^{(u)})$ have the same $\ast$-distribution so that $\mu_{w,z}$ is the distribution of 
$$
\left[P( \vert w \vert c^{(1)}, \ldots, \vert w \vert c^{(u)},a^{(1)},\dots,a^{(t)})-zI\right]\left[P( \vert w \vert c^{(1)}, 
\ldots, \vert w \vert c^{(u)},a^{(1)},\dots,a^{(t)})-zI\right]^*.
$$
Then the result follows from Proposition \ref{conjecture}.
\end{proof}
\begin{lemma}\label{rho}
Let $\Gamma$ be a compact subset in $\{z\in\mathbb{C}\colon0\notin\text{supp}(\mu_z)\}$. Then there 
exists $\rho>1$ such that for any $w\in \mathbb{C}$ such that $ \vert w \vert \leq \rho$ and any $z \in 
\Gamma$, we have $0\notin \text{supp}(\mu_{w,z}).$
\end{lemma}
\begin{proof}
Let $z$ be in $\Gamma$. According to Proposition \ref{conjecture}, $z\notin S_0$ and $\text{spect}(\Delta_1(z)) \subseteq \overline{\mathbb D}
\setminus\{1\}$. According to \cite[Theorem 2.5]{EH}, if $r(z)$ is the spectral radius of the  positive linear map $\Delta_1(z)$, then there exists a nonzero positive element
$\xi$ in ${\cal M}_m(\mathbb C)$ such that $\Delta_1(z)(\xi)=r(z)\xi$. Thus, we can deduce that $r(z)<1$.
Now, since $\{z\in\mathbb{C}\colon0\notin\text{supp}(\mu_z)\} \subset \C \setminus S_0$, using Remark \ref{supportloi} and Lemma \ref{inversible}, it is easy to see that  ($z \mapsto r(z)$)  is continuous on $\{z\in\mathbb{C}\colon0\notin\text{supp}(\mu_z)\}$. Thus,  there exists $0<\gamma<1$ such that for any $z\in \Gamma$, we have $0\leq r(z) <1-\gamma.$ It readily follows that if $\vert w \vert \leq \frac{1}{\sqrt{1-\gamma}}$ then $\vert w\vert^2 r(z) <1$ and according to Lemma \ref{w}, $0 \notin \text{supp}(\mu_{w,z})$.
\end{proof}
\begin{lemma}\label{svw}
Let $\Gamma$ be a compact subset in $\{z \in \mathbb{C}, 0 \notin supp(\mu_z)\}$. Assume that  $({\bf A'_2})$ holds. Then there exists $\rho>1$ and $\eta>0$ such that a.s. for all large $N$, for any $w\in \mathbb{C}$ such that $ \vert w \vert \leq \rho$ and any $z \in \Gamma$, there is no singular value of 
$$P\left( w\frac{X_N^{(1)}}{\sqrt{N}}, \ldots, w\frac{X_N^{(u)}}{\sqrt{N}},  (A_N^{(1)})', \ldots, (A_N^{(t)})'\right)-zI_N$$ in $[0,\eta]$.
\end{lemma}
\begin{proof}
Let $\tilde \Gamma=\{(w,z) \in \mathbb{C}^2, \vert w\vert \leq \rho, z\in \Gamma\}$ where $\rho$ is defined in Lemma \ref{rho}. According to Lemma \ref{rho}, $\forall (w,z) \in \tilde \Gamma$, $0\notin \text{supp}(\mu_{w,z})$. Therefore, using $({\bf A'_2})$,  according to Proposition \ref{nonunif}, there exists $\gamma( w,z)$ such that a.s. for all large N, there is no singular value of $$P\left( w\frac{X_N^{(1)}}{\sqrt{N}}, \ldots, w\frac{X_N^{(u)}}{\sqrt{N}},  (A_N^{(1)})', \ldots, (A_N^{(t)})'\right)-zI_N$$ in $[0,\gamma(w,z)]$. The conclusion follows by a compacity argument (using Bai-Yin's theorem and \eqref{normedeAprime}.
\end{proof}

\begin{lemma}\label{rayonspectral}
Let $\Gamma$ be a compact subset in $\{z\in\mathbb{C}\colon0\notin\text{supp}(\mu_z)\}$.  Assume that  $({\bf A'_2})$ and \eqref{normedeAprime} hold. There 
exists $0<\epsilon_0<1$ such that almost surely for all large $N$, we have for any $z$ in $\Gamma$,
$$
\rho \left(R'_{N}(z)  Y_N \right) \leq 1-\epsilon_0,
$$
where $\rho(M) $ denotes the spectral radius of a matrix $M$.
\end{lemma}

\begin{proof}

Now, assume that $\lambda\neq 0$ is an eigenvalue of $R'_{N}(z) Y_N$.  Then there exists $v \in \mathbb{C}^{Nm}$, $v\neq 0$ such that 
$(ze_{11} \otimes I_N -\gamma \otimes I_N -\sum_{k=1}^t\beta_k \otimes
(A_N^{(k)})^{'})^{-1} Y_Nv =\lambda v$ and thus $(ze_{11} \otimes I_N -\gamma \otimes I_N  -\sum_{k=1}^t\beta_k \otimes
(A_N^{(k)})^{'} -\sum_{j=1}^u \zeta_j \otimes \lambda^{-1} \frac{X_N^{(j)}}{\sqrt{N}}) v=0.$
This means that  $z$ is an eigenvalue of $$P\left( \lambda^{-1}\frac{X_N^{(1)}}{\sqrt{N}}, \ldots, \lambda^{-1}\frac{X_N^{(u)}}{\sqrt{N}},  (A_N^{(1)})', \ldots, (A_N^{(t)})'\right),$$ or equivalently that  $0$ is a singular value of
$$P\left( \lambda^{-1}\frac{X_N^{(1)}}{\sqrt{N}}, \ldots, \lambda^{-1}\frac{X_N^{(u)}}{\sqrt{N}},  (A_N^{(1)})', \ldots, (A_N^{(t)})'\right)-zI_N.$$

By Lemma \ref{svw}, we can deduce that almost surely for all large $N$, the nonnul eigenvalues of $R'_{N}(z)Y_N$ must satisfy $1 / | \lambda| >  \rho  $. The result follows.
\end{proof}

\noindent {\bf Step 3: Study of the moments of $R{'}_N(z)Y_N$}.
\begin{proposition}\label{borne} Let $\Gamma$ be a compact subset in $\{z \in \mathbb{C}, 0 \notin supp(\mu_z)\}$.  Assume that  $({\bf A'_2})$ and \eqref{normedeAprime} hold.
There exists $0<\epsilon_0<1$ and $C>0$ such that almost surely for all large $N$, for any $k \geq 1$,
$$\sup_{ z \in \Gamma} \left\|  \left(R'_{N}(z)  Y_N \right)^k\right\| \leq C(1-\epsilon_0)^k.$$
\end{proposition}
\begin{proof} For $z \in \Gamma$, we set 
$T_N(z)=R'_{N}(z)  Y_N .$
 Let $\epsilon_0$ be as defined by Lemma \ref{rayonspectral} and $\rho$ be as defined in Lemma \ref{svw}.
 Choose $0< \epsilon<\min( \epsilon_0, 1-\frac{1}{\rho})$. 
Therefore, according to Lemma \ref{rayonspectral} and using Dunford-Riesz calculus, we have almost surely for all large $N$,  for any $z$ in $\Gamma$,
$$\forall k \geq 0~~, (T_N(z))^k =\frac{1}{2i\pi} \int_{\vert w\vert = 1-\epsilon} {w^k}{(w-T_N(z))^{-1}} dw,$$
and therefore \begin{equation}\label{cauchy} \forall k \geq 0~~, \Vert  (T_N(z))^k \Vert \leq  \sup_{\vert w\vert = 1-\epsilon } \Vert {(w-T_N(z))^{-1}}\Vert {( 1-\epsilon ) }^{k+1}.\end{equation}

\noindent Now, note that, for  any $w$ such that $\vert w\vert = 1-\epsilon$, we have $\frac{1}{\vert w\vert}< \rho$ and   \\

 $(w-T_N(z))=$  $$w R'_N(z)  \left( ze_{11} \otimes I_N -\gamma \otimes I_N  -\sum_{j=1}^u \zeta_j \otimes w^{-1}\frac{X_N^{(j)}}{\sqrt{N}} -\sum_{k=1}^t\beta_k \otimes
(A_N^{(k)})^{'}  \right),$$
so that  \\

\noindent $(w-T_N(z))^{-1}=$  $$   \left( ze_{11} \otimes I_N -L_P( w^{-1}\frac{X_N^{(1)}}{\sqrt{N}}, \hspace*{-0.2cm}\ldots, w^{-1}\frac{X_N^{(u)}}{\sqrt{N}},  (A_N^{(1)})^{'}, \ldots, (A_N^{(t)})^{'})\right)^{-1}$$ \begin{equation}\label{decomp}\times \frac{1}{w} \left(ze_{11} \otimes I_N -\gamma \otimes I_N  -\sum_{k=1}^t\beta_k \otimes
(A_N^{(k)})^{'}\right).\end{equation}
Lemma \ref{svw} readily implies that almost surely for all large $N$, 
\begin{equation}\label{normeP} 
\left\|\left(zI_N - P( w^{-1}\frac{X_N^{(1)}}{\sqrt{N}}, \ldots, w^{-1}\frac{X_N^{(u)}}{\sqrt{N}},  (A_N^{(1)})^{'}, \ldots, (A_N^{(t)})^{'})\right)^{-1} \right\|\leq 1/\eta,\end{equation}
where $\eta$ is defined in Lemma \ref{svw}.\\
It readily follows from  \eqref{decomp}, Lemma \ref{resHari}, \eqref{normeP}, \eqref{normedeAprime} and Bai-Yin's theorem that 
 there exists $C>0$ such that  we have almost surely for all large $N$,  for any $z$ in $\Gamma$,
\begin{equation}\label{normeT} \sup_{\vert w\vert = 1-\epsilon } \Vert {(w-T_N(z))^{-1}}\Vert \leq C.\end{equation}
Proposition \ref{borne} follows from \eqref{cauchy} and \eqref{normeT}.\end{proof}

\noindent {\bf Step 4: Conclusion.}\\
We will use the following proposition from \cite{BoC} to   establish Lemma \ref{serie} below.
\begin{proposition}[\cite{BoC}]\label{bordenavecapitaine}
Let $n \geq 1$ be an integer and $Q \in \mathbb{C} \langle X_1, \cdots, X_n \rangle$ 
such that the total exponent of $X_n$ in each monomial of $Q$ is nonzero. We consider a sequence $(B_N^{(1)}, \cdots, B_{N}^{(n-1)}) \in {\cal M}_{N} (\mathbb{C})^{n-1}$ of matrices with operator norm uniformly bounded in $N$ and $u_N$, $v_N$ in $ {\mathbb{C}}^N$  with unit norm. Then if $X_N$ is a $N\times N$ matrix with iid entries centered with variance $1$ and finite fourth moment  a.s. 
$$
u_N ^* Q \left(  B_N ^ {(1)}, \cdots, B_N^{(n-1)} , \frac{X_N}{\sqrt{N}} \right) v_N \to 0. 
$$
\end{proposition}
\begin{lemma}\label{serie}  Assume (${\bf X_1}$), \eqref{normedeA}  and  (${\bf A_2}$). 
For any $z$ in $\Gamma\subset \mathbb{C}\setminus\mathrm{spect}(P(c,a))$, almost surely, 
the series $\sum_{k\geq 1}Q_N\left(R'_{N}(z) Y_N \right)^k R'_{N}(z)P_N$ converges in norm towards zero when $N$ goes to infinity, where $P_N$ and $Q_N$ are defined by \eqref{defPQ}.
\end{lemma}
\begin{proof}
The singular value decomposition of $\sum_k\beta_k \otimes
(A_N^{(k)})^{''} $ gives that for any $i,j \in \{1,\ldots,p\}$,
$$
(Q_N \left(R'_{N}(z)  Y_N \right)^k  R'_{N}(z) P_N)_{ij}= s_i v_i^* \left(R'_{N}(z)  Y_N \right)^k R'_{N}(z)u_j,
$$
where $u_j$ and $v_j$ are unit vectors in $\C^{Nm}$ and $s_i$ is a singular value of $\sum_{k=1}^t\beta_k\otimes
(A_N^{(k)})^{''}$. According to \eqref{normedeA} and \eqref{normedeAprime}, the $s_i$'s are uniformly bounded. Using
 (${\bf A_2'}$), \eqref{normedeAprime} and \eqref{Hari},  almost surely for 
any $z$ in $\Gamma$, there exists $\tilde \eta_z>0$ such that for all large $N$, 
\begin{equation}\label{majresA} 
\Vert R'_{N}(z)\Vert \leq \frac{1}{\tilde\eta_z}.
\end{equation}
Using \eqref{majresA} and  Bai-Yin's theorem,  we deduce from  Proposition \ref{bordenavecapitaine}  that $v_i^*(R'_{N}(z)Y_N )^kR'_{N}(z)u_j$ converges 
almost surely to zero. The result follows by applying the dominated convergence theorem thanks to 
Proposition \ref{borne}.
\end{proof}

We are going to  prove that, assuming  (${\bf X_1}$), \eqref{normedeA}  and  (${\bf A_2}$), we have  for any $z$ in $\Gamma$, almost surely, as $N \to \infty$, 
\begin{equation}\label{eq:convRz}
\Vert Q_N R_N(z) P_N -Q_N R'_{N}(z) P_N \Vert \to 0. 
\end{equation}
 Let $C'>0$ such that $\left\| P_N \right\| \left\| Q_N \right\| \leq C'$.  According to 
Proposition \ref{nonunif} 
 and \eqref{Hari},
  for any $z \in \Gamma$, there exists $\tilde \gamma_z>0$ such that   almost surely for all large $N$
\begin{equation}\label{majresA2}
\Vert R_N(z) \Vert \leq \frac{1}{\tilde \gamma_z}.\end{equation} Then 
using also Proposition \ref{borne} and \eqref{majresA}, for any $k\geq 1$, we have
$$\left\| Q_N  \left(R'_{N}(z)  Y_N \right)^kR'_{N}(z) P_N\right\| \leq \frac{CC'}{\tilde \eta_z}(1-\epsilon_0)^k,$$ 
$$\left\| Q_N  \left(R'_{N}(z)  Y_N \right)^kR_N(z) P_N\right\| \leq \frac{CC'}{\tilde \gamma_z}(1-\epsilon_0)^k.$$ 
Let $\eta>0$. Choose $K \geq 1$ such that $\frac{CC'}{\tilde \gamma_z}(1-\epsilon_0)^K< \eta/2$
and $\sum_{k\geq K}  \frac{CC'}{\tilde \eta_z}(1-\epsilon_0)^k <\eta/2$.\\

\noindent Thus, using \eqref{exp}, we have that,   for any $\eta>0$, 
$$ \left\|Q_NR_N(z) P_N - Q_N R'_N(z) P_N - \sum_{k\geq 1} Q_N  \left(R'_N(z)   Y_N \right)^kR'_N(z) P_N \right\| < \eta$$
and then,  letting $\eta$ going to zero, that
we have \begin{equation}\label{eq:devtaylor}
Q_N R_N(z) P_N  -  Q_N R'_N(z) P_N  = \sum_{k\geq 1} Q_N  \left(R'_N(z)  Y_N\right)^kR'_N(z) P_N.
\end{equation}
Applying Lemma \ref{serie}, we obtain \eqref{eq:convRz}. 
\begin{proposition}\label{diffdets} Let $\Gamma$ be a compact subset of $\mathbb{C} \setminus \text{spect}(P(c,a)).$  Assume (${\bf X_1}$), \eqref{normedeA}  and  (${\bf A_2}$). Then, almost surely, 
$\det ( I_p-Q_N R_N(z )P_N )-\det ( I_p-Q_N R'_N(z)P_N  )$  converges to zero uniformly on $\Gamma$, when $N$ goes to infinity.
\end{proposition}
\begin{proof}
It sufficient to check that for any $\delta >0$, a.s., for all large $N$, 
\begin{equation}\label{eq:convRzunif}
\sup_{ z \in \Gamma} \left\| Q_N R_N(z) P_N -Q_N R'_{N}(z) P_N \right\| \leq 3 \delta. 
\end{equation}
We set $\zeta_z =\tilde  \eta_z \wedge \tilde \gamma_z $ and $r_z =  ( \zeta_z /2 ) \wedge  ( \delta ( \zeta_z^2 / 2 C' ) )$. Using the resolvent identity, 
 \eqref{majresA} and \eqref{majresA2}, if $(z,w)\in \Gamma^2$ are such that $|z - w| \leq r_z$, then 
\begin{eqnarray*}
\| Q_N R_N (z) P_N  - Q_N R_N (w) P_N \| &\leq &\frac{ 2 C'}{\zeta_z^2}    |z - w| \leq \delta\\
\| Q_N R'_N (z) P_N  - Q_N R'_N (w) P_N \|& \leq &\frac{ 2 C'}{\zeta_z^2}  |z - w| \leq \delta
\end{eqnarray*}
Since $\Gamma \subset \cup_{z \in \Gamma} B( z, r_z)$ and $\Gamma$ compact, there is a finite covering and the proposition  follows from \eqref{eq:convRz}. 
\end{proof}

Theorem \ref{outlier} follows from Proposition \ref{diffdets}
by Rouch\'e's Theorem, using \eqref{quotient} and \eqref{conditionRouche}.


\end{document}